\documentclass[12pt]{article}

\usepackage[utf8]{inputenc}
\usepackage[english]{babel}
\usepackage{authblk}
\usepackage{lipsum}
\usepackage{amsfonts}
\usepackage{graphicx}
\usepackage{epstopdf}
\usepackage{algorithmic}
\usepackage{amsopn}
\usepackage{amssymb}
\usepackage{amsmath}
\usepackage{amsthm}
\usepackage{xcolor}
\usepackage{lineno}
\usepackage{hyperref}
\usepackage{geometry}
\usepackage{titlesec}
\usepackage{array} 
\usepackage{multirow} 
\usepackage{url}
\usepackage{nicefrac}
\usepackage{dsfont}
\usepackage{bbm}
\usepackage{enumitem}
\usepackage[normalem]{ulem}
\usepackage[labelfont=bf]{caption}

\hypersetup{colorlinks=true, linkcolor=blue, citecolor=blue}
\titleformat{\section}{\normalfont\large\bfseries}{\thesection.}{0.5em}{}
\titleformat{\subsection}{\normalfont\normalsize\bfseries}{\thesubsection.}{0.5em}{}
\titleformat{\subsubsection}{\normalfont\normalsize\bfseries}{\thesubsubsection.}{0.5em}{}
\numberwithin{equation}{section}
\numberwithin{figure}{section}
\DeclareGraphicsExtensions{.pdf,.png,.jpg}
\graphicspath{{pic/}}

\newtheorem{theorem}{Theorem}[section]

\newtheorem{lemma}[theorem]{Lemma}

\newtheorem{remark}[theorem]{Remark}
\newtheorem{definition}[theorem]{Definition}

\newcommand{\bnull}{\boldsymbol 0}
\renewcommand{\bf}{\boldsymbol f}
\newcommand{\bg}{\boldsymbol g}
\newcommand{\bn}{\boldsymbol n}

\newcommand{\bs}{\boldsymbol s}
\newcommand{\bu}{\boldsymbol u}

\newcommand{\bv}{\boldsymbol v}
\newcommand{\bw}{\boldsymbol w}

\newcommand{\bx}{\boldsymbol x}
\newcommand{\bA}{\boldsymbol A}

\newcommand{\btau}{\boldsymbol\tau}

\newcommand{\btheta}{\boldsymbol\theta}

\newcommand{\bxi}{\boldsymbol\xi}

\newcommand{\bnabla}{\boldsymbol\nabla}

\newcommand{\mC}{\mathbb C}

\newcommand{\mR}{\mathbb R}

\newcommand{\Tin}{T_{\mathrm{in}}}
\newcommand{\Tex}{T_{\mathrm{ex}}}
\newcommand{\Win}{W_{\mathrm{in}}}
\newcommand{\Wf}{W_{f}}
\newcommand{\WO}{W_{0}}
\newcommand{\HD}{H_{\mathcal{D}}}
\newcommand{\GD}{\Gamma_{\mathcal D}}
\newcommand{\GN}{\Gamma_{\mathcal N}}
\newcommand{\GF}{\Gamma_{\mathcal F}}
\newcommand{\GR}{\Gamma_{\mathcal R}}
\newcommand{\Dbox}{D_{\rm box}}
\newcommand{\tf}{{t_f}}

\DeclareMathOperator{\dx}{\operatorname{d}\!\bx}
\DeclareMathOperator{\dt}{\operatorname{d}\! t}
\DeclareMathOperator{\ds}{\operatorname{d}\!\bs}
\DeclareMathOperator{\C}{\operatorname{C}}
\DeclareMathOperator{\J}{\operatorname{J}}
\DeclareMathOperator{\F}{\operatorname{F}}

\DeclareMathOperator{\bI}{\operatorname{\bold I}}
\DeclareMathOperator{\bB}{\operatorname{\bold B}}

\renewcommand{\div}{\operatorname{div}}
\let\oldforall\forall
\renewcommand{\forall}{\oldforall\,}

\title{Shape optimization for thermoelasticity with temperature-dependent material parameters}
\author[$\ast$]{Marc Dambrine}
\author[$\dagger$]{Helmut Harbrecht}
\author[$\dagger$]{Viacheslav Karnaev}
\affil[$\ast$]{Universit\'e de Pau et des Pays de l’Adour, E2S UPPA, 
	CNRS, LMAP, UMR 5142, 64000 Pau, France.}
\affil[$\dagger$]{Departement Mathematik und Informatik, 
	Universit\"at Basel, \newline 4051 Basel, Switzerland}
\date{\today}

\begin{document}
	
\maketitle


\begin{abstract}
We consider the numerical solution of shape optimization 
problems for thermoelasticity with temperature-dependent 
material parameters. We show the existence of the 
shape derivative and derive an expression for generic 
functionals of domain integral type. Numerical results are 
presented for two settings: minimization of the compliance 
under a volume constraint, and minimization of the volume 
under a constraint on the \(L^2\)-norm of the von Mises 
stress. We use the finite element method for solving the 
underlying boundary value problems and the level set 
method for the representation of the actual domain.
\end{abstract}

\section{Introduction}
Mechanical components in the engines of cars, airplanes, 
or helicopters are exposed to significant heat fluxes and 
high temperatures generated by fuel combustion or the 
exhaust of hot gases. Depending on the operating mode 
of the engines, these heat flows can vary significantly. 
Typically, they are higher when the engine is under heavy 
load, for example at the startup, and lower when idling. 

In this article, we aim to incorporate insights into this 
cycle in the design of the structure in question. To this end, 
we will approach the problem from the perspective of 
shape optimization by including a priori information about 
the thermal behaviour in order to derive a volumetric 
shape that is optimized with respect to a predefined quality 
measure. Consequently, the shape optimization problem 
must take two phenomena into account: mechanical and 
thermal effects, though we neglect the heat sources 
generated by the mechanical phenomena. It must 
hence account for the temporal effects of the 
temperature of the external environment.

The mathematical model which underlies the structure 
of interest consists therefore of two coupled boundary 
value problems: the semilinear heat equation and the 
system of thermoelasticity. A weak coupling is assumed: 
the temperature affects the stress field, but mechanical 
deformations do not influence the temperature field. The 
material parameters are assumed to vary with the temperature, 
thereby influencing stress, strain, and heat transfer within the 
material. The elastic body is subjected to a given volume and 
surface forces. It is assumed to be embedded in an external 
medium with which it exchanges heat. The heat exchange is 
modeled via a Robin boundary condition, representing a 
heat flux across the boundary that is proportional to the 
temperature difference between the body and the surrounding 
medium. The mechanical loads are considered known, whereas 
the temperature field of the external medium is assumed to vary 
significantly in both space and time. 

Based on the given thermoelasticity model, we consider the 
optimization of the structure of interest in case of generic 
shape functionals of domain integral type. We compute the 
shape derivative of such functionals and propose a numerical 
shape optimization algorithm that combines the level set 
method to represent the actual domain with the finite 
element method to compute the state and its adjoint. 
We illustrate our approach by optimizing the compliance 
and the $L^2$-norm of the von Mises stress under a 
prescribed volume constraint.

The rest of the article is organized as follows. Section~\ref{sec:problem}
introduces the notation used thereafter and states the thermoelastic
model under consideration as well as the constrained shape optimization 
problem we consider. In Section~\ref{sct:calculus}, we then compute 
the shape derivative of the functionals under consideration. The
numerical realization by the level set method and the finite element 
method is presented in Section~\ref{sec:numerics}. Two numerical 
experiments are performed to illustrate and validate the approach.
Finally, in Section~\ref{sec:conclusio}, we draw the article's 
conclusion.

\section{Problem formulation}\label{sec:problem}
The objective of this section is to introduce the notation and 
the precise formulation of the thermoelasticity model under 
consideration. In contrast to \cite{AJ,DAMBRINE2025}, we 
consider here a nonlinear model which also takes thermal
effects in the material parameters into account. Finally, we 
define the shape optimization problem we want to solve.

\subsection{Notation}
First, we introduce some general notation. Let  \(D\subset 
\mR^d,\ d=2,3\), be a bounded and connected domain with 
smooth boundary \(\partial D\), which is divided into three 
subsets \(\GD\), \(\GN\) and \(\GF\) satisfying
\[
|\GD|, |\GN|, |\GF| > 0\quad\text{such that}\quad
\partial D = \GD\cup\GN\cup\GF.
\]
Hereinafter, we denote by \(\bn\) the outward pointing unit 
normal vector on \(\partial D\) and set \(\GR:=\GN\cup\GF\).

We next introduce the time interval \((0,\tf)\), 
where \(\tf=\text{const}>0\) and consider a 
thermoelastic body, represented by the region 
\(D\). The state of the body is determined by 
the scalar field \(T\) of temperature and the vector field 
\(\bu\) of displacements:
\[
	T(t,\mathbf{x}) \colon (0,\tf)\times D\to \mR, 
	\quad \bu(t,\bx) \colon (0,\tf)\times D\to \mR^d.
\] 

Throughout this article, we use the following notation 
for the deformation tensor 
\[
	\varepsilon(\bu):=\frac{1}{2}(\boldsymbol\bnabla \bu
	+\boldsymbol\bnabla\bu^\top), \quad\text{where}\quad  
	[\boldsymbol\bnabla \bu]_{i,j} := \partial_{x_j} u_i
\]
and, according to the Duhamel-Neumann postulate, 
for the stress tensor
\begin{equation}\label{eq:stress_tnesor}
	\sigma(T, \bu) := \mC(T):\varepsilon(\bu) + (T-\Tin)\bB(T),
\end{equation}
where 
\[
	\mC(T):\varepsilon(\bu) = 2\mu(T)\varepsilon(\bu) 
	+  \lambda(T)\div(\bu)\bI\quad\text{and}\quad\bB(T) 
	= -\alpha(T)\big(3\lambda(T) + 2\mu(T)\big)\bI.
\]
Here, \(\bI\) is the identity matrix, \(\Tin\geq 0\) is the initial 
temperature and the Lam\'e constants are
\[
	\mu(T)=\frac{E(T)}{2\big(1+\nu(T)\big)}\quad\text{and}
\quad\lambda(T)=\frac{E(T)\nu(T)}{\big(1+\nu(T)\big)\big(1-2\nu(T)\big)}.
\]

The properties of the body are completely characterized by 
the symmetric fourth-order stiffness tensor \(\mC(T)\) and the 
matrix \(\bB(T)\), which both depend on the temperature 
by means of temperature dependent material parameters: 
Young's modulus \(E(T) \in C^1(\mR;\mR^+)\), Poisson's ratio 
\(\nu(T)\in C^1\big(\mR; (-1,1/2)\big)\), and the thermal expansion 
coefficient \(\alpha(T)\in C^1(\mR;\mR^+)\). In what follows, 
we denote \(\mR^+ := \{x \in \mR \mid x > 0\}\). Additionally, 
we assume that \(\mC(T)\) satisfies the uniform ellipticity 
condition, i.e., there exist constants \(C_1, C_2 > 0\) such that
\begin{equation}\label{eq:ellipticity_condition}
C_{1}|\bxi|^2 \leq \mC(s):\bxi:\bxi\leq	C_{2}|\bxi|^2  
\quad \forall 0\neq\bxi\in\mR^{d\times d}_\text{sym}, \, s\in\mR. 
\end{equation}

The heat exchange is governed by the thermal conductivity 
\( k(T) \in C^2(\mR;\mR^+) \), the specific heat capacity 
\( \widetilde C(T) \in C^1(\mR;\mR^+) \), and the constant 
mass density \( \widetilde\rho > 0 \) and heat transfer coefficient 
\( \beta > 0 \). We denote the product 
\( \rho(T) := \widetilde{\rho}\,\widetilde{C}(T) \).
Furthermore, we assume that there exist constants 
\(\rho_1,\rho_2,\rho'_1,\rho'_2>0\) and \(k_1,k_2,k'_1,k'_2>0\) 
such that for all \(s \in \mathbb{R}\) it holds
\begin{equation}\label{eq:parabolicity_condition}		
	\begin{gathered}
	 	\rho_1 \leq \rho(s) \leq \rho_2,\quad
		k_1 \leq k(s) \leq k_2, \\
		\rho'_1 \leq \rho'(s) \leq \rho'_2,\quad
		k'_1 \leq k'(s) \leq k'_2.
	\end{gathered}
\end{equation}

Finally, we shall introduce the spaces of admissible solutions
for the boundary value problems under consideration. To this 
end, we define the Bochner spaces of time-dependent 
\(H^1\)-smooth functions: 
\[
	H\big((0,\tf),D\big) := L^2\big((0,\tf); H^1(D)\big)
\]
and
\[
	\HD\big((0,\tf),D\big) := \big\{ u \in H\big((0,\tf),D\big)
	\ | \  u = 0 \text{~on~} (0,\tf)\times\GD\big\}.
\]  
The space \(\HD\big((0,\tf),D\big)^d\) serves as the energy space 
for the displacement field. 

We next introduce the anisotropic space 
\begin{align*}
	W\big((0,\tf),D\big) 
	&:=  L^2\big((0,\tf); H^1(D)\big)
	\cap H^1\big((0,\tf); H^{-1}(D)\big) \\
	&=\big\{T \in  L^2\big((0,\tf); H^1(D)\big) \ | \ 
	\partial_t T \in L^2\big((0,\tf); H^{-1}(D)\big)\big\},
\end{align*}
which we equip with the graph norm. It gives rise
to the energy space for the temperature by means of
\begin{align*}
	\Win\big((0,\tf),D\big) := \big\{ T \in W\big((0,\tf),D\big)  \ | \
	T = \Tin \text{~on~}& (0,\tf)\times \GD \\[1ex]
	&\text{and~} T = \Tin \text{~in~} \{t=0\}\times D \big\}.
\end{align*}
Moreover, we need two more spaces for the local shape 
derivative of the temperature, which incorporates a 
homogeneous initial condition
\begin{equation}\label{eq:w0_space}
	\begin{aligned}
		\WO\big((0,\tf),D\big) := \big\{ T \in W\big((0,\tf),D\big)  \ | \
		T = 0 \text{~on~}& (0,\tf)\times \GD \\[1ex]
		&\text{and~} T = 0 \text{~in~} \{t=0\}\times D \big\},
	\end{aligned}
\end{equation}
and for the adjoint state of the temperature, which 
incorporates a homogeneous terminal condition
\begin{equation}\label{eq:wf_space}
	\begin{aligned}
		\Wf\big((0,\tf),D\big) := \big\{ T \in W\big((0,\tf),D\big)  \ | \
		T = 0 \text{~on~}& (0,\tf)\times \GD \\[1ex]
		&\text{and~} T = 0 \text{~in~} \{t=\tf\}\times D \big\}.
	\end{aligned}
\end{equation}
 
\subsection{Governing equations}
We consider the following model of the thermoelastic body \(D\) 
which combines the nonlinear heat equation with the equations 
of thermoelasticity. The mechanical unknowns of the model are 
the temperature field \(T \in \Win\big((0,\tf),D\big)\) and the displacement 
field \(\bu \in \HD\big((0,\tf),D\big)^d\), which are described by the following 
equations.

\begin{description}
	\item[\(\bullet\)~\textbf{Heat equation.}] 
	\begin{equation}\label{eq:heat}
		\left\{\,
		\begin{aligned}
			\rho(T)\partial_t T - \div\big(k(T)\bnabla T\big) 
			&= Q &&\text{in~~} (0,\tf)\times D, \\[1ex]
			k(T)\bnabla T\cdot \bn + \beta T &= \beta\Tex
			&&\text{on~~} (0,\tf)\times\GR. \\[1ex]
		\end{aligned}
		\right.
	\end{equation}
	where the thermal exchange with the environment is taken into 
	account through the Robin boundary conditions with the external
	temperature \(\Tex \in L^2\big((0,\tf); L^2(\GR)\big)\) such that \(\Tex(0) = \Tin\).
	Additionally, the body is subject to a thermal source \(Q \in L^2\big((0,\tf); 
	L^2(D)\big)\), which can also be time-dependent.
\end{description}
\begin{description}
	\item[\(\bullet\)~\textbf{Thermoelasticity equilibrium system.}] 
	\begin{equation}\label{eq:elasticity}
		\left\{
		\begin{aligned}
			-\div(\sigma(T, \bu)) &= \bf &&\text{in~~} 
			(0,\tf)\times D, \\[1ex]
			\sigma(T, \bu) \bn &= \bg &&\text{on~~} 
			(0,\tf)\times \GN, \\[1ex]
			\sigma(T, \bu) \bn &= \bnull &&\text{on~~} 
			(0,\tf)\times \GF.\\[1ex]
		\end{aligned}
		\right.
	\end{equation}
	where \(\bf\in L^2\big((0,\tf); L^2(D)^d\big)\) is the body force acting 
	throughout the domain \(D\), and \(\bg \in L^2\big((0,\tf); L^2(\GN)^d\big)\) 
	is the surface traction applied on the boundary part \(\GN\). 
	The body is assumed to be fixed on \(\GD\) and free 
	on the remaining part \(\GF\). The illustration of the 
	model can be found in Figure \ref{fig:model}.
\end{description}

\begin{figure}[hbt]
	\centering
	\includegraphics[width=0.4\linewidth]{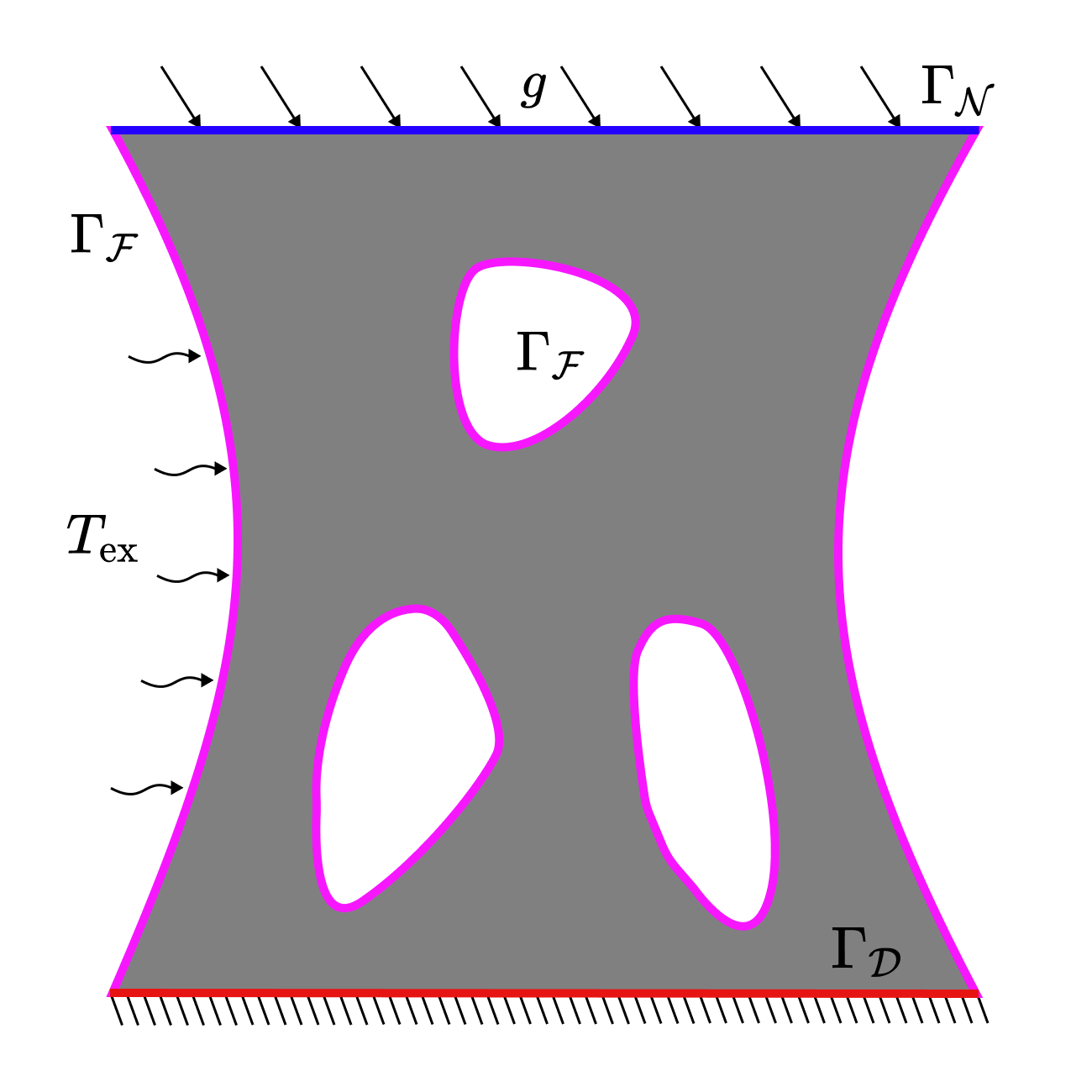}
	\caption{Illustration of the model for the thermoelastic body.}
	\label{fig:model}
\end{figure}

The variational formulation of the problem
\eqref{eq:heat} reads as follows: find \(T \in \Win\big((0,\tf); D\big)\) 
such that
\begin{equation}\label{eq:var_heat}
	\begin{aligned}
		\int_0^\tf \int_D \rho(T)\partial_t T S\dx\dt + 
		 \int_0^\tf\int_Dk(T)\bnabla T\cdot\bnabla S\dx\dt 
		+ \int_0^\tf\int_{\GR} \beta T S\ds\dt&  \\[1ex]
		= \int_0^\tf\int_{D} QS\dx\dt + \int_0^{\tf}\int_{\GR} 
		\beta\Tex S\ds\dt \quad \forall S\in \HD\big((0,\tf),D\big).&
	\end{aligned}
\end{equation}
The variational formulation of the problem \eqref{eq:elasticity} 
reads: find \(\bu\in \HD\big((0,\tf),D\big)^d \) such that 
\begin{equation}\label{eq:var_elasticity}
	\begin{aligned}
		\int_0^\tf\int_D\sigma(T, \bu) : \varepsilon(\bv)\dx\dt 
		= 	\int_0^\tf\int_D \bf\cdot\bv\dx\dt
		+ 	\int_0^\tf\int_{\GN} &\bg\cdot\bv\ds\dt\\[1ex]
		&\forall  \bv \in \HD\big((0,\tf),D\big)^d 
	\end{aligned}
\end{equation}

Under the assumption \eqref{eq:parabolicity_condition}, 
the variational equation \eqref{eq:var_heat} under consideration 
admits a unique solution \(T \in \Win\big((0,\tf),D\big)\), see 
\cite[Thm.\ 8.9]{roubivcek2005nonlinear}. Note that it can 
be solved independently of \eqref{eq:var_elasticity}. Thus, 
according to the definition \eqref{eq:stress_tnesor} of the 
stress tensor, the term \((T-\Tin)\bB(T)\) appears as a forcing 
term in \eqref{eq:var_elasticity}. Consequently, under the 
assumption \eqref{eq:ellipticity_condition}, the quasi-static 
problem \eqref{eq:var_elasticity} also admits a unique solution 
\(\bu \in \HD\big((0,\tf),D\big)^d\), see \cite[Thm.\ 4.4]{boccardo2013elliptic}. 
By applying Green’s formulas, it is easy to verify that the solutions 
to the variational problems \eqref{eq:var_heat} and 
\eqref{eq:var_elasticity} are equivalent to the strong formulations 
\eqref{eq:heat} and \eqref{eq:elasticity}, respectively.
Especially, one can show that
\begin{equation}\label{eq:temp_regularity}
	T\in L^\infty\big((0,\tf); H^1(D)\big) \cap L^2\big((0,\tf); H^2(D)\big) 
	\quad\text{and}\quad\partial_t T \in L^2\big((0,\tf); L^2(D)\big),
\end{equation}
see \cite[Chpt.\ 8.8.3]{roubivcek2005nonlinear}.

\subsection{Shape optimization problem}
We are interested in the optimal shape of a given body described 
by the model above. By optimal shape, we mean the one that 
minimizes an objective functional \(J(D)\) under prescribed 
constraints \(C(D)\) which we define as integral functionals 
of general form:
\[
	\J(D) := \int_0^\tf\int_Dj(T_D, \bnabla T_D, \bu_D,
	\bnabla\bu_D)\dx\dt
\]
and
\[
	\C(D) := \int_0^\tf\int_D c(T_D, \bnabla T_D, \bu_D,
	\bnabla\bu_D)\dx\dt.
\]
Here, \(T_D\in\Win\big((0,\tf),D\big)\) and \(\bu_D\in \HD\big((0,\tf), D\big)^d\) 
denote the solutions to \eqref{eq:var_heat} and \eqref{eq:var_elasticity}, 
respectively, on the domain \(D\). The functions \(j(\cdot)\) and \(c(\cdot)\) 
are required to be at least of class \(C^1\) with respect to the state 
variables and their gradients, and that they belong
\(L^1\big((0,\tf); L^1(D)\big)\).

We shall consider the minimization of the objective function 
\(\J(D)\) with a constraint on \(\C(D)\) not to exceed a desired
constant threshold \(\tau>0\). It is assumed that we are looking 
for an optimal structure \(D\) which is contained in some bounded 
reference domain \(\Dbox\). Thus, the problem formulation becomes
\begin{equation}\label{eq:det_opt_problem}
	\underset{D\subset\Dbox}{\text{minimize}}\quad \J(D)
	\quad\text{subject to}\quad \C(D)\leq \tau.
\end{equation}
For the sake of simplicity, we consider that the boundaries 
\(\GD\) and \(\GN\) are non-optimizable, i.e.~fixed, which is also 
reasonable from the application point of view.

\begin{remark}
It should be noted that shape optimization problems are often 
ill-posed. We refer the reader to \cite[Sct.~3.1]{allaire2021shape} 
for an instructive example of non-existence of solutions. Even 
in cases where the shape optimization problem lacks a solution, 
there is still a significant practical value. Engineers frequently 
want to develop a component design that is better, but also 
close to the current one without necessarily achieving optimality.
\end{remark}

\section{Shape calculus}\label{sct:calculus}
Our next focus is the computation of the shape derivative
of the functionals in the optimization problem 
\eqref{eq:det_opt_problem}  under consideration. We use the 
traditional method, in which the deformations of \(\partial D\) 
are parameterized by means of a vector field 
\(\btheta \in W^{1,\infty}(\Dbox; \mR^d)\) and then the functionals 
are differentiated in the Fréchet sense with respect to \(\btheta\). 
To this end, we introduce Lagrangian differentiation and 
calculate the derivatives of the temperature \(T_D\) and the 
displacement field \(\bu_D\). Then, we establish the shape 
derivative for a generic shape functional of domain integral
type. For details about shape calculus, we refer to \cite{allaire2021shape,
henrot2018shape,sokolowski1992introduction} and the 
references therein.
	
\subsection{Basic identities}
First, we need to recall the concept of shape differentiability.	
Let us consider a vector field \(\btheta\in W^{1,\infty}(\Dbox; \mR^d)\) 
such that 
\begin{itemize}
\item the mapping \((\text{Id} + \btheta)\) diffeomorphically 
takes the domain \(\Dbox\) onto itself,
\item there holds \(\btheta = \bnull\) on \(\GD\cup\GN\).
\end{itemize}
We define the family \(D_{\btheta}\) of domains 
by setting \(D_{\btheta} := (\text{Id}+\btheta)(D)\). 


\begin{definition}[Lagrangian derivative]\label{def:lagrange_deriv} 
Let \(D \subset \mR^d\) be a bounded, Lipschitz domain and
\(T_D:(0,\tf)\times D\to\mR\) be an associated function. 
The mapping \(T_D\) has a Lagrangian derivative at a particular 
shape \(D\) if the transported function
\[
\widehat{T}_D(\btheta) := T_{D_{\btheta}}\circ 
(\text{{\rm Id}}+\btheta)
\]
is Fr\'echet differentiable at \(\btheta  = \bnull\). Its 
Fr\'echet derivative \(\dot T_D (\btheta)\) is called 
the Lagrangian derivative of \(T_D\).
\end{definition}
A similar definition can be introduced for a vector field 
\(\bu_D:(0,\tf)\times D\to\mR^d\).

In the subsequent computations, we shall frequently 
use the surface divergence of a vector field, which 
is defined by \(\div_{\btau}(\btheta) := \div(\btheta) - 
(\bnabla\btheta\bn)\cdot\bn\). Moreover, we denote 
by \(J(\btheta) := |\det(\bI + \bnabla \btheta)|\) 
the Jacobian determinant associated with the mapping 
\((\text{Id}+\btheta)\). 

With this at hand, we can formulate the following lemma
which characterizes the Lagrangian derivative of the 
temperature field.

\begin{lemma}\label{lem:deriv_temp}
	Let \(D\subset \mR^2\), then there exists the Lagrangian derivative 
	\(\dot T_D(\btheta)\in W_0\big((0,\tf), D\big)\) of the solution 
	\(T_D\in \Win\big((0,\tf), D\big)\) to \eqref{eq:heat}. It satisfies
	the boundary value problem
	\begin{equation}\label{eq:deriv_temp}
		\left\{\;\begin{aligned}
			&\rho(T_D)\partial_t\dot T_D(\btheta)
			-\div\big(k(T_D) \bnabla \dot T_D(\btheta)\big) 
			- \div\big(k'(T_D)\dot T_D(\btheta)\bnabla T_D\big)\\
			&\hspace{52mm}+\rho'(T_D)\partial_t T_D \dot T_D(\btheta) 
			= \widetilde{F}_D(\btheta)&&\text{~~in~~} (0,\tf)\times D, \\[1ex]
			&k(T_D) \bnabla \dot T_D(\btheta)\cdot \bn  = 0  
			 &&\text{~~on~~} (0,\tf)\times\GN, \\[1ex]
 			&k(T_D) \bnabla \dot T_D(\btheta)\cdot \bn 
			 + \beta\dot T_D(\btheta)+ \beta\dot T_D(\btheta)
			 =  \widetilde{G}_D(\btheta)  &&\text{~~on~~} (0,\tf)\times \GF, \\[1ex]
		\end{aligned}\right.
	\end{equation}
	where 
	\begin{equation}\label{eq:deriv_temp_not}
		\begin{gathered}
				\widetilde{F}_D(\btheta) =  \div(Q\btheta) 
				- \div(\btheta)\rho(T_D)\partial_t T_D 
				+ \div\big(k(T_D)\bA'(\btheta)\bnabla T_D\big) , \\[1ex]
				\widetilde{G}_D(\btheta) = \beta\big(\div_{\btau}(\Tex\btheta) - \div_{\btau}(\btheta) T_D\big)
				- k(T_D)\big(\bA'(\btheta)\bnabla T_D\big)\cdot \bn,\\[1ex]
				\bA'(\btheta) := \div(\btheta)\bI - \bnabla\btheta 
				- (\bnabla\btheta)^\top.
		\end{gathered}
	\end{equation}
\end{lemma}

\begin{proof}
	We present the proof in three steps.
	
	{\it Step 1.} First, we derive the variational formulation of \eqref{eq:heat} 
	for the transported mapping \(\widehat{T}_D(\btheta) 
	= T_{D_{\btheta}} \circ (\text{Id} + \btheta)\). To achieve this, we 
	start from the associated variational formulation \(\eqref{eq:var_heat}\) 
	on the perturbed domain \(D_{\btheta}\): 
	\begin{align*}
		\int_0^\tf \int_{D_{\btheta}}\rho(T_{D_{\btheta}})
		\partial_t T_{D_{\btheta}} S\dx\dt 
		+ \int_0^\tf\int_{D_{\btheta}} k(T_{D_{\btheta}})
		\bnabla T_{D_{\btheta}}\cdot\bnabla S\dx\dt 
		+ \int_0^\tf\int_{{\GR}_{\btheta}} \beta T_{D_{\btheta}} 
		S\ds\dt& \\[1ex]
		=\int_0^\tf\int_{D_{\btheta}} QS\dx\dt
		+ \int_0^\tf\int_{{\GR}_{\btheta}} \beta\Tex S\ds\dt
		\quad \forall S\in H_{\mathcal{D}}\big((0,\tf),D_{\btheta}\big).&
	\end{align*}
	Note that only the part \(\GF\) of the boundary is perturbed, 
	since \(\btheta = \bnull\) on \(\GD\cup\GN\). We 
	transport this formulation back to the original domain \(D\)
	by using the chain rule and a change of variables:
	\begin{equation}\label{eq:deriv_temp_proof}
		\begin{aligned}
			\int_0^\tf \int_D&\rho\big(\widehat{T}_D(\btheta)\big)
			\partial_t \widehat{T}_D(\btheta) S
			J(\btheta)\dx\dt \\[1ex]
			&+ \int_0^\tf \int_D k\big(\widehat{T}_D(\btheta)\big)\bA(\btheta)
			\bnabla\widehat{T}_D(\btheta)\cdot\bnabla S\dx\dt 
			+ \int_0^\tf \int_{\GR} \beta a(\btheta)
			\widehat{T}_D(\btheta) S\ds\dt \\[1ex] 
			&\hspace{12mm}= \int_0^\tf \int_D\ \widehat{Q}(\btheta)S
			J(\btheta)\dx\dt +  \int_0^\tf \int_{\GR} \beta 
			a(\btheta)\widehat{T}_\mathrm{ex}(\btheta) S \ds\dt \\[1ex]
			&\hspace{96mm}\forall S\in \HD\big((0,\tf),D\big),
		\end{aligned}
	\end{equation}
	where
	\[
		\bA(\btheta) := J(\btheta)(\bI 
		+\bnabla\btheta)^{-1}(\bI +\bnabla\btheta)^{-\top}
	\]
		and
	\[
		a(\btheta) := \|(\bI +\bnabla\btheta)^{-1}\bn\| 
		J(\btheta).
	\]
	
	We next introduce the operator 
	\[
	\mathcal{A}(\btheta,T): W^{1,\infty}(\Dbox; \mR^d)\times 
	W\big((0,\tf),D\big) \to L^2\big((0,\tf);H^{-1}(D)\big),
	\]
	defined by
	\begin{align*}
		\mathcal{A}(\btheta,T)\langle S \rangle := 	
		\int_0^\tf \int_D&\rho(T)\partial_t T S
		J(\btheta)\dx\dt\\[1ex] 
		&+ \int_0^\tf \int_D k(T)\bA(\btheta)\bnabla T
		\cdot\bnabla S\dx\dt + \int_0^\tf \int_{\GR}\beta 
		a(\btheta)T S\ds\dt, 
	\end{align*}
	and \(b(\btheta):  W^{1,\infty}(\Dbox; \mR^d) \to 
	L^2\big((0,\tf);H^{-1}(D)\big)\), defined by
	\[
		b(\btheta)\langle S \rangle := 
		\int_0^\tf \int_D\ \widehat{Q}(\btheta)S
		J(\btheta)\dx\dt + \int_0^\tf \int_{\GR} 
		\beta a(\btheta) \widehat{T}_\mathrm{ex}(\btheta) S\ds\dt.
	\]
	Thus, we can rewrite the identity \eqref{eq:deriv_temp_proof} 
	as 
	\[
		\mathcal{F}\big(\btheta,\widehat{T}_D(\btheta)\big)\langle S \rangle 
		= 0 \quad \forall S \in \HD\big((0,\tf),D\big),
	\]
	where
	\[
		\mathcal{F}(\btheta,T) := \mathcal{A}(\btheta,T) - b(\btheta).
	\]
	Hence, the partial derivative 
	\[
		\partial_T\mathcal{F}(0,T_D)
		: W\big((0,\tf),D\big) \to L^2\big((0,\tf);H^{-1}(D)\big) 
	\] 
	is characterized by the linear operator 
	\begin{align*}
		\partial_T \mathcal{F}(0,T_D)\langle Z, S \rangle 
		= &\int_0^\tf\int_D\big(\rho'(T_D)Z\partial_t T_D
		+ \rho(T_D)\partial_t Z\big)S\dx\dt \\[1ex]
		&+\int_0^\tf \int_D\big(k'(T_D) Z \bnabla T_D
		\cdot\bnabla S + k(T_D)\bnabla Z\cdot\bnabla S \big)\dx\dt \\[1ex]
		&+ \int_0^\tf\int_{\GR} \beta Z S\ds\dt.
	\end{align*}
	
	{\it Step 2.} Our next goal is to prove that the operator 
	\(\partial_T\mathcal{F}(0,T_D)\) is an isomorphism. 
	To this end, we introduce the bilinear form
	\begin{equation}\label{eq:deriv_temp_proof4}
		\begin{aligned}
			\mathcal{B}_t(t)\langle Z, S \rangle  := 
			&\int_D\big(\rho'(T_D)Z\partial_t T_DS
			+ k'(T_D) Z \bnabla T_D
			\cdot\bnabla S + k(T_D)\bnabla Z\cdot\bnabla S \big)(t)\dx \\[1ex]
			&+ \int_{\GR} \beta Z(t) S(t)\ds.
		\end{aligned}
	\end{equation}
	and set
	\begin{equation}\label{eq:BLF}
		\mathcal{B}\langle Z, S \rangle := 
		\int_0^{\tf} \mathcal{B}_t(t)\langle Z, S \rangle\dt = 
		\partial_T\mathcal{F}(0,T_D)\langle Z, S \rangle.
	\end{equation}
	
	By the Cauchy-Schwarz inequality and the assumptions 
	\eqref{eq:parabolicity_condition}, we obtain
	\begin{equation}\label{eq:deriv_temp_proof1}
		\begin{aligned}
			\Bigg|\int_0^\tf\int_D&\big(\rho(T_D)\partial_t ZS 
			+ k(T_D)\bnabla Z\cdot\bnabla S\big)\dx\dt\Bigg|  \\[1ex]
			&\leq \rho_2\|\partial_t Z
			\|_{L^2((0,\tf); H^{-1}(D))}\|S\|_{L^2((0,\tf); H^1(D))}\\[1ex]
			&\hspace{10mm} + k_2   \|Z\|_{L^2((0,\tf); H^1(D))} 
			\|S\|_{L^2((0,\tf); H^1(D))}\\[1ex]
			&\leq C_1\|Z\|_{\Win((0,\tf), D)} \|S\|_{\HD((0,\tf),D)},
		\end{aligned}
	\end{equation}
	where \(C_1 = \max\{\rho_2,k_2\}\). 
	Employing the assumptions \eqref{eq:parabolicity_condition}, 
	the regularity results \eqref{eq:temp_regularity}, and the
	inequalities of H\"older as well as of Ladyzhenskaya, we get 
	\begin{equation}\label{eq:deriv_temp_proof2}
		\begin{aligned}
			\Bigg|\int_0^\tf\int_D&\big(\rho'(T_D)\partial_t T_DS 
			+ k'(T_D)\bnabla T_D\cdot\bnabla S\big)Z\dx\dt\Bigg| \\[1ex]
			&\leq \rho'_2 \|\partial_t T_D\|_{L^2((0,\tf)\times D)} 
			\|Z\|_{L^4((0,\tf)\times D)}\|S\|_{L^4((0,\tf)\times D)} \\[1ex]
			&\hspace{10mm} + k'_2 \|\bnabla T_D\|_{L^4((0,\tf)\times D)} 
			\|Z\|_{L^4((0,\tf)\times D)}\|\bnabla S\|_{L^2((0,\tf)\times D)} \\[1ex]
			&\leq C_L\rho'_2 \|\partial_t T_D\|_{L^2((0,\tf);L^2(D))}
			\|Z\|_{\Win((0,\tf), D)} \|S\|_{\HD((0,\tf),D)} \\[1ex]
			&\hspace{10mm} + C_L k'_2 \|\bnabla T_D\|_{L^4((0,\tf)\times D)} 
			\|Z\|_{\Win((0,\tf), D)}\|S\|_{\HD((0,\tf),D)} \\[1ex]
			&\leq C_2 \|Z\|_{\Win((0,\tf), D)}\|S\|_{\HD((0,\tf),D)}, 
		\end{aligned}
	\end{equation}
	where  
	\[
		C_2 = C_L\max\Big\{\rho'_2 
		\|\partial_t T_D\|_{L^2((0,\tf);L^2(D))}, 
		k'_2 \|T_D\|^{1/2}_{L^\infty((0,\tf);H^1(D))}
		\|T_D\|^{1/2}_{L^2((0,\tf);H^2(D))}\Big\}.
	\]
	Note that \(C_L\) depends only on \(D\) and \(\tf\). 
	By the Cauchy-Schwarz inequality and the trace 
	theorem, we conclude
	\begin{equation}\label{eq:deriv_temp_proof3}
		\begin{aligned}
		\Bigg|\int_0^\tf\int_{\GR} \beta ZS\ds\dt\Bigg| &\leq\beta
		\|Z\|_{L^2((0,\tf);L^2(\GR))}\|S\|_{L^2((0,\tf);L^2(\GR))}\\[1ex]
		&\leq \beta C_{\rm tr} \|Z\|_{\Win((0,\tf),D)} \|S\|_{\HD((0,\tf),D)}
		\end{aligned}
	\end{equation}
	with \(C_{\rm tr}>0\) denoting the constant in the 
	trace inequality. Finally, combining the estimates 
	\eqref{eq:deriv_temp_proof1}--\eqref{eq:deriv_temp_proof3}, 
	we conclude that the bilinear form $\mathcal{B}$
	from \eqref{eq:BLF} is bounded. 
	
	By the assumptions \eqref{eq:parabolicity_condition} and Poincar\'e's 
	inequality, we further conclude
	\begin{equation}\label{eq:deriv_temp_proof5}
		\begin{aligned}
			\int_Dk(T_D)|\bnabla Z(t)|^2\dx + \int_{\GR} \beta Z^2(t)\ds 
			\geq C_Pk_1\|Z(t)\|^2_{H^1_0(D)},
		\end{aligned}
	\end{equation}
	where constant \(C_P\) depends only on \(D\). In view of
	the assumptions \eqref{eq:parabolicity_condition}, H\"older’s, 
	Ladyzhenskaya’s and Young's inequalities, we derive
	\begin{equation}\label{eq:deriv_temp_proof6}
		\begin{aligned}
			\Bigg|\int_D\big( &\rho'(T_D)\partial_t T_DZ^2 + k'(T_D) Z \bnabla T_D
			\cdot\bnabla Z \big)(t)\dx\Bigg| \\[1ex]
			&\leq\rho'_2\|\partial_t T_D(t)\|_{L^2(D)} \|Z(t)\|^2_{L^4(D)} \\[1ex]
			&\hspace{10mm} + k'_2 \|Z(t)\|_{L^4(D)}  
			\|\bnabla T_D(t)\|_{L^4(D)} \|\bnabla Z(t)\|_{L^2(D)} \\[1ex]
			&\leq C_{LD}\rho'_2\|\partial_t T_D(t)\|_{L^2(D)} \|Z(t)\|_{L^2(D)} 
			\|\bnabla Z(t)\|_{L^2(D)}\\[1ex]
			&\hspace{10mm} + C_{LD}k'_2 \|Z(t)\|^{1/2}_{L^2(D)}  
			\|\bnabla T_D(t)\|_{L^4(D)} \|\bnabla Z(t)\|^{3/2}_{L^2(D)}\\[1ex]
			&\leq \epsilon\|\bnabla Z(t)\|^{2}_{L^2(D)} + C_{\epsilon} (C_{LD} \rho'_2)^2
			\|\partial_t T_D(t)\|^2_{L^2(D)} \|Z(t)\|^2_{L^2(D)} \\[1ex]
			&\hspace{10mm} + \epsilon \|\bnabla Z(t)\|^{2}_{L^2(D)} + 
			C_{\epsilon} (C_{LD} k'_2)^4 \|Z(t)\|^2_{L^2(D)}  \|\bnabla T_D(t)\|^4_{L^4(D)} \\[1ex]
			&\leq 2\epsilon\|\bnabla Z(t)\|^{2}_{L^2(D)} 
			+  C_3(t) \|Z(t)\|^2_{L^2(D)}.
		\end{aligned}
	\end{equation}
	Herein, \(C_3(t) := C_{\epsilon}(C_{LD} \rho'_2)^2\|\partial_t T_D(t)\|^2_{L^2(D)} 
	+ C_{\epsilon}(C_{LD} k'_2)^4\|\bnabla T_D(t)\|^4_{L^4(D)}\), where \(C_{LD}\) 
	depends only on \(D\) and \(C_\epsilon\) depends only on \(\epsilon\). By
	standard regularity results, we can further conclude that \(C_3(t)\in L^1(0,\tf)\). 
	Thus, by applying \eqref{eq:deriv_temp_proof5} and \eqref{eq:deriv_temp_proof6} 
	to \eqref{eq:deriv_temp_proof4} with \(0<\epsilon<C_Pk_1/2\), we obtain
	\begin{equation}\label{eq:deriv_temp_proof7}
	\begin{aligned}
		\mathcal{B}_t(t)(Z,Z) 
		&\geq (C_Pk_1-2\epsilon)\|Z(t)\|^2_{H^1_0(D)} - C_3(t) \|Z(t)\|^2_{L^2(D)}\\
		&\geq -C_3(t) \|Z(t)\|^2_{L^2(D)}.
	\end{aligned}
	\end{equation}
	
	Since \(\mathcal B\) is bounded and the associated spatial form
	$\mathcal{B}_t$ satisfies the inequality \eqref{eq:deriv_temp_proof7}, 
	the standard variational theory for linear (see, e.g., \cite[Sct.\ 7.1.2]{evans2022partial})
	parabolic equations implies for arbitrary \(f\in L^2\big((0,\tf);H^{-1}(D)\big)\) that
	there exists a unique \(Z\in \Win\big((0,\tf),D\big)\) such that
	\[
		\mathcal{B}\langle Z, S \rangle  = \langle f, S \rangle _{H^{-1}(D),H^1(D)} 
		\quad \forall S \in \HD\big((0,\tf), D\big).
	\]
	Consequently, \(\partial_T\mathcal F(0,T_D)\) is bijective. The
	corresponding energy estimate gives the boundedness of the inverse.
	Hence, \(\partial_T\mathcal F(0,T_D)\) is an isomorphism and the 
	Fr\'echet differentiability follows from the implicit function theorem, 
	see e.g.\ \cite[Chpt.\ I, Thm.\ 5.9]{LANG}.
	
	{\it Step 3.} In order to obtain an expression for the derivative \(\dot{T}_D(\btheta)\), 
	we calculate
	\[
		\partial_T\mathcal{A}(\bnull, T_D)
		\langle\dot{T}_D(\btheta), S\rangle
		= \partial_{\btheta}b(\bnull)\langle\btheta, S\rangle
		- \partial_{\btheta}\mathcal{A}(\bnull,T_D)
		\langle\btheta, S\rangle \quad \forall S \in \HD\big((0,\tf), D\big).
	\]
	For any \(\btheta \in  W^{1,\infty}(\Dbox; \mR^d)\) satisfying 
	\(\btheta = \bnull\) on \(\GD \cup \GN\), the corresponding 
	derivatives have the explicit expressions
	\begin{align*}
		\partial_T\mathcal{A}(\bnull, T_D)
		\langle\dot{T}_D(\btheta),S\rangle= &\int_0^\tf\int_D
		\big(\rho'(T_D)\dot{T}_D(\btheta) \partial_t T_D
		+ \rho(T_D)\partial_t\dot{T}_D(\btheta)\big)S\dx\dt \\[1ex]
		&+\int_0^\tf \int_D\big(k'(T_D)\dot{T}_D(\btheta) \bnabla T_D
		\cdot\bnabla S + k(T_D)\bnabla \dot{T}_D(\btheta)\cdot\bnabla S \big)\dx\dt \\[1ex]
		&+ \int_0^\tf\int_{\GF} \beta \dot{T}_D(\btheta) S\ds\dt,
	\end{align*}
	\begin{align*}
		\partial_{\btheta}\mathcal{A}(\bnull, T_D)
		\langle\btheta, S\rangle = &\int_0^\tf \int_D\div(\btheta)\rho(T_D)
		\partial_t T_D S\dx\dt \\[1ex]
		&+ \int_0^\tf\int_D k(T_D)\bA'(\btheta)\bnabla T_D\cdot\bnabla S\dx\dt \\[1ex]
		&+ \int_0^\tf\int_{\GF} \beta\div_{\btau}(\btheta) T_D S\ds\dt,
	\end{align*}
	and
	\[
		 \partial_{\btheta}b(\bnull)\langle\btheta, S\rangle
		 = \int_0^\tf\int_D \div(Q\btheta) S\dx\dt
		+ \int_0^\tf\int_{\GF} \beta\div_{\btau}(\Tex\btheta) S\ds\dt,
	\]
	where \(\bA'(\btheta)\) is defined in \eqref{eq:deriv_temp_not}.
	Finally we conclude the variational identity for \(\dot{T}_D(\btheta)
	\in \WO\big((0,\tf),D\big)\):
	\begin{equation}\label{eq:var_deriv_temp}
	\begin{aligned}
		&\int_0^\tf\int_D\rho(T_D)\partial_t\dot{T}_D(\btheta) S\dx\dt
		+\int_0^\tf\int_D k(T_D)\bnabla \dot{T}_D(\btheta)\cdot\bnabla S\dx\dt \\[1ex]
		&\qquad+\int_0^\tf\int_D \dot{T}_D(\btheta)
			\big(\rho'(T_D)\partial_t T_DS
			+ k'(T_D)\bnabla T_D\cdot\bnabla S\big)\dx\dt
		+\int_0^\tf\int_{\GF} \beta \dot{T}_D(\btheta)S\ds\dt\\[1ex]
		&\quad= \int_0^\tf\int_D\big(\div(Q\btheta)S 
			-\div(\btheta)\rho(T_D)\partial_t T_D S 
			- k(T_D)\bA'(\btheta)\bnabla T_D\cdot\bnabla S\big)\dx\dt \\[1ex]
		&\qquad+ \int_0^\tf\int_{\GF} \beta\big(\div_{\btau}(\Tex\btheta) - \div_{\btau}(\btheta) T_D\big)S\ds\dt 
			\quad\forall S\in \HD\big((0,\tf),D\big).
		\end{aligned}
	\end{equation}

	Finally, the claim follows by applying Green's formula 
	to \eqref{eq:var_deriv_temp} and taking into account 
	the boundary conditions from \eqref{eq:heat}.
\end{proof}

The next lemma characterizes the Lagrangian derivative of 
the displacement field.

\begin{lemma}\label{lem:deriv_disp}
	If the solution \(T_D\in W\big((0,tf), D\big)\) to \eqref{eq:heat} 
	is Lagrangian differentiable, then  the Lagrangian derivative 
	\(\dot \bu_{D}(\btheta) \in\HD\big((0,\tf),D\big)^d\) of the 
	solution \(\bu_D\in\HD\big((0,\tf), D\big)^d\) to 
	\eqref{eq:elasticity} satisfies
	\begin{equation}\label{eq:deriv_disp}
		\left\{\;\begin{aligned}
			-\div\big(\mC(T_D):\varepsilon(\dot\bu_D(\btheta))\big) 
			&=\widetilde{\boldsymbol{f}}_D(\btheta) &&\text{in~~} (0,\tf) \times D, \\[1ex]
			\big(\mC(T_D):\varepsilon(\dot\bu_D(\btheta))\big) \bn  &=  \bnull 
			&&\text{on~~} (0,\tf) \times \GN, \\[1ex]
			\big(\mC(T_D):\varepsilon(\dot\bu_D(\btheta))\big) \bn 
			&= \widetilde{\boldsymbol{g}}_D(\btheta) &&\text{on~~} (0,\tf) \times \GF,
		\end{aligned}\right.
	\end{equation}
	where \(T_D\in \Win\big((0,\tf), D\big)\) is the solution of 
	\eqref{eq:var_heat} and
	\begin{equation}\label{eq:deriv_disp_not}
		\begin{gathered}
			\hspace{-5mm}\widetilde{\boldsymbol{f}}_D(\btheta) 
			:= \div\!\Big(\!\div(\btheta)\sigma(T_D,\bu_D) 
			- \mC(T_D) :(\bnabla\btheta^\top\bnabla\bu_D) 
			- \bnabla\btheta\big(\mC(T_D) : \varepsilon(\bu_D)\big)\\
			\hspace{62mm} + \bf\otimes\btheta 
			+ \dot T_D(\btheta)\bB(T_D) 
			+ \dot T_D(\btheta)\sigma^\ast(T_D, \bu_D)\Big), \\[1ex]
			\hspace{-11mm}\widetilde{\boldsymbol{g}}_D(\btheta) 
			:= \Big(\!\div(\btheta)\sigma(T_D,\bu_D) 
			- \mC(T_D) :(\bnabla\btheta^\top\bnabla\bu_D) 
			- \bnabla\btheta\big(\mC(T_D) : \varepsilon(\bu_D)\big)\\
			\hspace{62mm} + \bf\otimes\btheta 
			+ \dot T_D(\btheta)\bB(T_D) 
			+ \dot T_D(\btheta)\sigma^\ast(T_D, \bu_D)\Big)\bn, \\[1ex]
			\sigma^\ast(T_D,\bu_D) 
			= \mC'(T_D):\varepsilon(\bu_D)+(T_D-\Tin)\bB'(T_D),  \\[1ex]
			\mC'(T_D):\varepsilon(\bu_D) = 2\mu'(T_D)\varepsilon(\bu_D) 
			+  \lambda'(T_D)\div(\bu_D)\bI, \\[1ex]
			\bB'(T_D) = -\alpha'(T_D)\big(3\lambda(T_D) + 2\mu(T_D)\big)\bI 
			-  \alpha(T_D)\big(3\lambda'(T_D) + 2\mu'(T_D)\big)\bI.
		\end{gathered}
	\end{equation}
	It admits the following variational identity for \(\dot\bu(\btheta)\in 
	\HD\big((0,\tf),D\big)^d\):
		\begin{equation}\label{eq:var_deriv_disp}
		\begin{aligned}
			\int_0^\tf\int_D&\mC(T_D):\varepsilon(\dot\bu_D) : 
			\varepsilon(\bv)\dx\dt \\[1ex]
			&= \int_0^\tf\int_D\Big(\mC(T_D) : 
			(\bnabla\btheta^\top\bnabla\bu_D) : \varepsilon(\bv) 
			+ \mC(T_D) : \varepsilon(\bu_D) : (\bnabla\btheta^\top\bnabla\bv) \\[1ex]
			&\hspace{25mm} - \div(\btheta) \sigma(T_D,\bu_D): 
			\varepsilon(\bv) - \dot T_D(\btheta)\bB(T_D) : \varepsilon(\bv) \\[1ex]
			&\hspace{25mm} - \dot T_D(\btheta) \sigma^\ast(T_D, \bu_D) : \varepsilon(\bv) 
			+ (T_D-\Tin)\big(\bB(T_D)\bnabla\btheta^\top\big):\bnabla\bv \\[1ex]
			&\hspace{25mm}+ \div(\bf\otimes\btheta)\cdot\bv\Big) \dx\dt 
			\quad \forall \bv\in\HD\big((0,\tf), D\big)^d.
		\end{aligned}
	\end{equation}
\end{lemma}

\begin{proof}
	The proof follows the same arguments as in the proof of
	Lemma \ref{lem:deriv_temp} and is therefore omitted.
\end{proof}

\begin{remark}\label{rem:deriv_3d}
	The proof of Lemma~\ref{lem:deriv_temp} is carried out in the
	two-dimensional setting. In order to extend the same argument to
	three spatial dimensions, additional regularity of the temperature
	field is required. More precisely, the estimates used in the proof 
	can be adapted to the three-dimensional case when assuming
	\[
	T_D\in L^\infty\big((0,\tf);H^2(D)\big)
	\cap W^{1,\infty}\big((0,\tf);L^2(D)\big).
	\]
\end{remark}

\begin{remark}\label{rem:deriv_existence}
	The well-posedness of the linear parabolic problem
	\eqref{eq:deriv_temp} can be established by the same arguments
	and estimates as those used in the proof of Lemma~\ref{lem:deriv_temp}.
	It allows us to apply the standard well-posedness theory for linear 
	parabolic equations; see, e.g., \cite[Sct.~7.1.2]{evans2022partial}.
	Hence, the problem \eqref{eq:deriv_temp} admits a unique weak
	solution. The linearized thermoelasticity problem \eqref{eq:deriv_disp} 
	is a linear elliptic problem with bounded and coercive bilinear form. 
	Therefore, by the Lax-Milgram theorem, it admits a unique weak 
	solution under assumptions \eqref{eq:ellipticity_condition}.
\end{remark}

\subsection{Shape derivative}
The objective of this subsection is the computation 
of the shape derivatives for the functionals of interest.
Their differentiability is defined in accordance with 
the following definition, compare \cite{DEZ,henrot2018shape, 
sokolowski1992introduction}.

\begin{definition}[Fr\'echet differentiable shape functional]
	\label{def:shape_deriv} 
	A shape functional \(\F(D)\) is Fr\'echet differentiable at 
	a given domain \(D\) if there exists a continuous linear 
	function \(\F'(D):  W^{1,\infty}(\Dbox; \mR^d)\to \mR\) such that
	\[
		\F(D_{\btheta}) = \F(D) + \F'(D)\langle\btheta\rangle + o(\btheta)
	\]
	for all \(\btheta \in  W^{1,\infty}(\Dbox; \mR^d)\). The linear form 
	\(\F'(D)\langle\cdot\rangle\) is called the shape derivative of \(\F\!\) in \(D\).
\end{definition}

In the context of unconstrained shape optimization, the 
shape derivative is employed to identify a direction 
\(\btheta\) of deformation such that \(\F'(D)(\btheta)< 0\). 
This direction of deformation serves then as a descent 
direction in a suitable optimization algorithm, allowing 
for the minimization of the objective functional \(\F(D)\). 

Before proceeding to the shape derivatives of the functionals, 
we need to mention an important theorem (see  
\cite[Prop.\ 5.9.1]{henrot2018shape}).

\begin{theorem}[Hadamard’s structure theorem]\label{th:hadamard}
	Let $D\subset\Dbox$ be a $C^1$-smooth domain. We suppose 
	that $F : \Dbox \to \mR$ is a differentiable functional in the 
	sense of Definition~\ref{def:shape_deriv}. If $\btheta\cdot\bn = \bnull$ 
	on the boundary $\partial D$, then there holds $F'(D)\langle\btheta\rangle = 0$.
\end{theorem}

The application of this theorem is explained in the following remark. 

\begin{remark}\label{rem:hadamard}
	In the case of a sufficiently regular domain \(D\), we can 
	conclude from this theorem that the value of the derivative 
	\(F'(D)\langle\btheta\rangle \) depends only on the normal 
	component of the vector field \(\btheta\) on the boundary 
	\(\partial D\), i.e.
	\[
		F'(D)\langle\btheta\rangle = \int_{\partial D} v_D(\btheta\cdot\bn)\ds,
	\]
	where \(v_D:\partial D\to\mR\) is a scalar field whose 
	expression depends on the solutions of the underlying
	boundary value problems and the functional form. Thus, 
	in the case of an unconstrained optimization problem, a 
	descent direction \(\btheta\) to the optimal shape is easily 
	obtained by imposing that \(\btheta = -v_D\bn\) on \(\partial D\). 
	Consequently, we have
	\[
		F'(D)\langle-v_D\bn\rangle  = -\int_{\partial D} v_D^2\ds< 0.
	\]
\end{remark}

We are now in a position to derive the shape derivative of 
a generic shape functional of domain integral type within the 
framework of thermoelasticity with temperature dependent 
material parameters. The result is presented in the following 
theorem.

\begin{theorem}\label{th:shape_deriv}
Given \(f_D\in L^1\big((0,\tf);L^1(D)\big)\), let \(F(D)\) 
be a functional defined by
\[
\F(D) :=  \int_0^\tf \int_D f_D\dx\dt \quad\text{with}\quad
f_D := f(T_D,\bnabla T_D,\bu_D,\bnabla\bu_D),
\]
where \(T_D \in \Win\big((0,\tf),D\big)\) and
\(\bu_D\in \HD\big((0,\tf),D\big)^d\) denote the solutions to 
\eqref{eq:heat} and \eqref{eq:elasticity} on the domain \(D\), 
respectively. If \(f(\cdot)\) is continuously differentiable and
its first order partial derivatives, evaluated at the state 
variables, belong to \(L^2\big((0,\tf);L^2(D)\big)\), and if the state 
variables are Lagrangian differentiable, then \(F(D)\) is 
Fr\'echet differentiable and its shape derivative is given by
\begin{equation}\label{eq:shape_deriv}
	\begin{aligned}
			\F'(D)\langle\btheta\rangle &=\int_0^\tf\int_{\GF} \Big(f_D 
			- \sigma(T_D,\bu_D) : \varepsilon(\bw_D) + \bf\cdot\bw_D\\[1ex]
			 &\hspace{29mm}- \rho(T_D)\partial_t T_D R_D 
			 - k(T_D)\bnabla T_D\cdot\bnabla R_D + QR_D  \\[1ex]
			&\hspace{29mm}+\beta(\mathcal{H} - \beta/k(T_D))(\Tex - T_D)R_D
			\Big)\Big(\btheta\cdot\bn\Big)\ds\dt,
		\end{aligned}
\end{equation}
where the function \(\bw_D\in \HD((0,tf),D)^d\) satisfies the adjoint system
\begin{equation}\label{eq:adjoint_disp}
	\left\{\;
	\begin{aligned}
		-\div\big(\mC(T_D):\varepsilon(\bw_D)\big) 
		&= -\div\big(\partial_{\bnabla\bu} f_D\big)+\partial_{\bu} f_D
		&&\text{in}\quad (0,\tf)\times D, \\[1ex]
		\big(\mC(T_D):\varepsilon(\bw_D)\big)\bn &= \partial{\bnabla\bu} f_D 
		\bn &&\text{on}\quad (0,\tf)\times \GR,
	\end{aligned}
	\right.
\end{equation}
while \(R_D\in \Wf\big((0,\tf),D\big)\) satisfies the backward-in-time adjoint system
\begin{equation}\label{eq:adjoint_temp}
	\left\{\;
	\begin{aligned}
		\rho(T_D)\partial_t R_D+\div\big(k(T_D) \bnabla R_D\big) 
		- \rho'(T_D)\partial_t T_DR_D&\\[1ex]
		- k'(T_D)\bnabla T_D\cdot\bnabla R_D 
		&= \widetilde{Q}&&\text{in}\quad (0,\tf)\times D, \\[1ex]
		\big(k(T_D) \bnabla R_{D}\big)\cdot \bn  
		&=  \partial_{\bnabla T} f_D\cdot\bn
		&&\text{on}\quad (0,\tf)\times\GN, \\[1ex]
		\big(k(T_D) \bnabla R_{D}\big)\cdot \bn + \beta R_D 
		&= \partial_{\bnabla T} f_D\cdot\bn
		&&\text{on}\quad (0,\tf)\times\GF,
	\end{aligned}
	\right.
\end{equation}
where 
\[
	\widetilde{Q} := \div\big(\partial_{\bnabla T} f_D\big) 
	-\partial_T f_D
	+ \sigma^\ast(T_D, \bu_D):\varepsilon(\bw_D) 
	+ \bB(T_D):\varepsilon(\bw_D)
\] 
and \( \sigma^\ast(T_D, \bu_D)\) is defined in \eqref{eq:deriv_disp_not}. 
\end{theorem}

\begin{proof}
	We present the computation of \(\F'(D)\) in three steps.
	
	{\it Step 1.} For sufficiently small 
	\(\btheta \in W^{1,\infty}(\Dbox; \mR^d)\), a change of variables 
	in the shape functional \(\F(D_{\btheta})\) yields
	\begin{align*}
			\F(D_{\btheta})
			&= \int_0^\tf\int_{D_{\btheta}}f_{D_{\btheta}}\dx\dt \\[1ex]
			&=\int_0^\tf\int_D f\big(\widehat T_D(\btheta), (\bI + \bnabla\btheta)^{-\top}
			\bnabla\widehat T_D(\btheta), \widehat\bu_D(\btheta),
			(\bI +\bnabla\btheta)^{-\top}\bnabla\widehat\bu_D(\btheta)\big)
			J(\btheta)\dx\dt.
		\end{align*}
	By taking the derivative using Definitions \ref{def:lagrange_deriv} 
	and \ref{def:shape_deriv} in the above formula, we obtain 
	\begin{equation}\label{eq:shape_deriv_proof1}
			\begin{aligned}
					\F'(D)\langle\btheta\rangle
					&=\int_0^\tf\int_D \Big(\div(\btheta)f_D 
					+ \partial_T f_D\dot T_D(\btheta) 
					+ \partial_{\bnabla T} f_D\cdot\big(\bnabla\dot T_D(\btheta) 
					- \bnabla\btheta^\top\bnabla T_D\big)\\
					&\hspace{24mm} + \partial_{\bu} f_D \cdot\dot \bu_D(\btheta) 
					+ \partial_{\bnabla\bu} f_D : \big(\bnabla\dot \bu_D(\btheta) 
					- \bnabla\btheta^\top\bnabla \bu_D\big)\Big)\dx\dt.
				\end{aligned}
		\end{equation}
		
	{\it Step 2.} In order to remove the dependence on \(\dot\bu_D(\btheta)\),
	we formulate the variational identity for the adjoint 
	state \(\bw_D \in \HD\big((0,\tf), D\big)^d\) in accordance with
	\begin{equation}\label{eq:var_adj_disp}
		\begin{aligned}
			\int_0^\tf\int_D \mC(T_D):\varepsilon(\bw_D):\varepsilon(\bv)\dx\dt
			= \int_0^\tf\int_D\Big(\partial_{\bu} f_D\cdot\bv + 
			\partial_{\bnabla\bu} f_D:\bnabla\bv\Big)\dx\dt& \\[1ex]
			\forall\bv\in\HD\big((0,\tf), D\big)^d.&
		\end{aligned}
	\end{equation}
	By using Green's formula, it is straightforward to verify that the 
	variational identity amounts to the system \eqref{eq:adjoint_disp}.
	Taking \(\dot\bu_D(\btheta)\) as test function in 
	\eqref{eq:var_adj_disp} and \(\bw_D\) as test function in 
	\eqref{eq:var_deriv_disp}, we conclude 
	\begin{align*}
		&\int_0^\tf\int_D\Big(\partial_{\bu} f_D\cdot\dot\bu_D(\btheta) 
		+ \partial_{\bnabla\bu} f_D:\bnabla\dot\bu_D(\btheta)\Big)\dx\dt\\[1ex]
		&\hspace{15mm}= \int_0^\tf\int_D \Big(\mC(T_D):(\bnabla\btheta^\top\bnabla\bu_D) 
		: \varepsilon(\bw_D) + \mC(T_D) : \varepsilon(\bu_D) 
		: (\bnabla\btheta^\top\bnabla\bw_D) \\[1ex]
		&\hspace{40mm} - \div(\btheta) \sigma(T_D,\bu_D):\varepsilon(\bw_D) 
		- \dot T_D(\btheta)\bB(T_D):\varepsilon(\bw_D) \\[1ex]
		&\hspace{40mm}- \dot T_D(\btheta) \sigma^\ast(T_D, \bu_D) : \varepsilon(\bw_D) 
		+ (T_D-\Tin)\big(\bB(T_D)\bnabla\btheta^\top\big):\bnabla\bw_D\\[1ex] 
		&\hspace{40mm}+ \div(\bf\otimes\btheta)\cdot\bw_D\Big)\dx\dt.
	\end{align*}
	With this identity at hand, we can transform 
	\eqref{eq:shape_deriv_proof1} into 
	\begin{equation}\label{eq:shape_deriv_proof2}
		\begin{aligned}
				\F'(D)\langle\btheta\rangle =\int_0^\tf&\int_D \Big(\div(\btheta)f_D 
				+ \partial_T f_D\dot T_D(\btheta) 
				+ \partial_{\bnabla T} f_D\cdot\big(\bnabla\dot T_D(\btheta) 
				- \bnabla\btheta^\top\bnabla T_D\big)\\[1ex]
				&\quad\qquad- \partial_{\bnabla\bu} f_D : 
				(\bnabla\btheta^\top\bnabla \bu_D) + \mC(T_D) : 
				(\bnabla\btheta^\top\bnabla\bu_D) : \varepsilon(\bw_D) \\[1ex] 
				&\quad\qquad+\mC(T_D):\varepsilon(\bu_D) : 
				(\bnabla\btheta^\top\bnabla\bw_D) - \div(\btheta) 
				\sigma(T_D,\bu_D):\varepsilon(\bw_D) \\[1ex]
				&\quad\qquad - \dot T_D(\btheta)\bB(T_D) : \varepsilon(\bw_D) 
				- \dot T_D(\btheta) \sigma^\ast(T_D, \bu_D):\varepsilon(\bw_D) \\[1ex]
				&\quad\qquad + (T_D-\Tin)\big(\bB(T_D)\bnabla\btheta^\top\big):\bnabla\bw_D
				+\div(\bf\otimes\btheta)\cdot\bw_D \Big)\dx\dt.
			\end{aligned}
	\end{equation}
	
	In order to eliminate \(\dot T_D(\btheta)\) in \eqref{eq:shape_deriv_proof2}, 
	we introduce a variational identity for the adjoint state \(R_D\in \Wf\big((0,\tf),D\big)\) 
	which is given by the variational formulation
	\begin{equation}\label{eq:var_adj_temp}
		\begin{aligned}
				&\int_0^\tf\int_D\rho(T_D)\partial_t R_D S\dx\dt
				-\int_0^\tf\int_D k(T_D)\bnabla R_D\cdot\bnabla S\dx\dt \\[1ex]
				&\quad-\int_0^\tf\int_D \Big(\rho'(T_D)\partial_t T_D R_D
				+ k'(T_D)\bnabla T_D\cdot\bnabla R_D\Big)S\dx\dt 
				-\int_0^\tf\int_{\GF} \beta R_DS\ds\dt\\[1ex]
				&\qquad= -\int_0^\tf\int_D \Big(\partial_T f_D S
				+ \partial_{\bnabla T} f_D\cdot\bnabla S
				- S\bB(T_D):\varepsilon(\bw_D)  \\[1ex] 
				&\hspace{47mm}-S \sigma^\ast(T_D, \bu_D) : \varepsilon(\bw_D) \Big) \dx\dt
				\quad\forall S \in \HD\big((0,\tf),D\big).
			\end{aligned}
	\end{equation}
	By means of Green's formula, it follows that this 
	variational identity implies the system \eqref{eq:adjoint_temp}.
	By taking \(\dot T_D(\btheta)\) as test function in 
	\eqref{eq:var_adj_temp} and \(R_D\) as test function 
	in \eqref{eq:var_deriv_temp}, applying integration by parts 
	and using the boundary conditions from \eqref{eq:heat} and 
	\eqref{eq:deriv_temp}, we arrive at
	\begin{align*}
		&\int_0^\tf\int_D \Big(\partial_T f_D\dot T_D(\btheta) 
		+ \partial_{\bnabla T} f_D\cdot\bnabla \dot T_D(\btheta) 
		-\dot T_D(\btheta)\bB(T_D):\varepsilon(\bw_D) \\
		&\hspace{72mm} - \dot T_D(\btheta) \sigma^\ast(T_D, \bu_D):\varepsilon(\bw_D) \bigg)\dx\dt  \\
		&\hspace{10mm} = \int_0^\tf\int_D\Big( \div(Q\btheta)R_D - \div(\btheta)\rho(T_D)\partial_t T_D R_D 
		- k(T_D)\bA'(\btheta)\bnabla T_D\cdot\bnabla R_D\Big)\dx\dt \\
		&\hspace{25mm}+ \int_0^\tf\int_{\GF} \beta\Big(\div_{\btau}(\Tex\btheta) - \div_{\btau}(\btheta) T_D\Big)R_D\ds\dt. 
	\end{align*}
	Thus, in view of equation \eqref{eq:shape_deriv_proof2}, 
	we derive the expression
	\begin{equation}\label{eq:shape_deriv_proof3}
		\begin{aligned}
				\F'(D)\langle\btheta\rangle &=\int_0^\tf\int_D \Big(\div(\btheta)f_D 
				- \partial_{\bnabla T} f_D\cdot
				(\bnabla\btheta^\top\bnabla T_D)- \partial_{\bnabla\bu}f_D : 
				(\bnabla\btheta^\top\bnabla \bu_D)\\[1ex] 
				&\hspace{25mm}+\mC(T_D):(\bnabla\btheta^\top\bnabla\bu_D) : 
				\varepsilon(\bw_D) + \mC(T_D):\varepsilon(\bu_D) : 
				(\bnabla\btheta^\top\bnabla\bw_D) \\[1ex]
				&\hspace{25mm}-\div(\btheta)  \sigma(T_D,\bu_D):\varepsilon(\bw_D) + (T_D-\Tin)\big(\bB(T_D)\bnabla\btheta^\top\big):\bnabla\bw_D\\[1ex]
				&\hspace{25mm} - \div(\btheta)\rho(T_D)\partial_t T_D R_D 
				- k(T_D)\bA'(\btheta)\bnabla T_D\cdot\bnabla R_D\\[1ex]
				&\hspace{25mm}+ \div(\bf\otimes\btheta)\cdot\bw_D+\div(Q\btheta)R_D \Big)\dx\dt \\[1ex]
				&\hspace{14mm} + \int_0^\tf\int_{\GF} \beta\Big(\div_{\btau}(\Tex\btheta) - \div_{\btau}(\btheta) T_D\Big)R_D\ds\dt. 
		\end{aligned}
	\end{equation}
	
	{\it Step 3.} Now we want to express the derivative in surface form. 
	Thanks to Theorem \ref{th:hadamard} and Remark \ref{rem:hadamard}, 
	we can consider that  \(\btheta = (\btheta\cdot\bn)\bn\) on \(\GF\). Therefore, 
	there holds \(\div_{\btau}(\btheta) = \mathcal{H}(\btheta\cdot\bn)\)
	and \(\div_{\btau}(\Tex\btheta)=  \mathcal{H}\Tex(\btheta\cdot\bn)\). 
	By applying Green's formula to the identity \eqref{eq:shape_deriv_proof3},
	we thus derive the boundary integral form of \(\F'(D)\) that is given by
	\begin{equation}\label{eq:shape_deriv_proof5}
		\begin{aligned}
			\F'(D)\langle\btheta\rangle &=\int_0^\tf\int_{\GF} \Big(f_D 
			- \partial_{\bnabla T} f_D\bn\cdot(\bnabla T_D\cdot\bn)
			-\partial_{\bnabla\bu} f_D\bn\cdot(\bnabla \bu_D\bn) \\[1ex]
			&\hspace{29mm}+\big(\mC(T_D):\varepsilon(\bw_D)\big)\bn \cdot (\bnabla\bu_D\bn)\\[1ex]
			&\hspace{29mm}+\big(\mC(T_D):\varepsilon(\bu_D)\big)\bn \cdot (\bnabla\bw_D\bn)  \\[1ex] 
			&\hspace{29mm}-\sigma(T_D,\bu_D):\varepsilon(\bw_D)\\[1ex]
			&\hspace{29mm}+ (T_D-\Tin)\bB(T_D)\bn\cdot(\bnabla\bw_D\bn)\\[1ex] 
			&\hspace{29mm} + k(T_D)(\bnabla T_D\cdot\bn)(\bnabla R_D\cdot\bn)\\[1ex]  
			&\hspace{29mm} - \rho(T_D)\partial_t T_D R_D - k(T_D)\bnabla T_D\cdot\bnabla R_D  \\[1ex] 
			&\hspace{29mm}+ \bf\cdot\bw_D + QR_D + \beta\mathcal{H}(\Tex - T_D)R_D\Big)\Big(\btheta\cdot\bn\Big)\ds\dt.
		\end{aligned}
	\end{equation}
	Inserting the boundary conditions in \eqref{eq:heat}, \eqref{eq:elasticity}, 
	\eqref{eq:adjoint_disp}, and \eqref{eq:adjoint_temp}, we obtain the 
	desired expression of \(\F'(D)\) from \eqref{eq:shape_deriv_proof5},
	which completes the proof.
\end{proof}

\begin{remark}
	The assumptions of Theorem~\ref{th:shape_deriv} require that the
	growth of the integrand \(f\) is compatible with the regularity of the
	state variables. In particular, the objective functional has to be
	well defined and the first order partial derivatives of \(f\), evaluated 
	at \((T_D,\nabla T_D,\bu_D,\nabla\bu_D)\), have to possess the
	integrability required by the adjoint variational problems.
	
	Furthermore, the differentiability result of Theorem~\ref{th:shape_deriv}
	relies on the existence of the Lagrangian derivatives of the state
	variables. The differentiability of the temperature and displacement
	fields is established in Lemmas~\ref{lem:deriv_temp} and
	\ref{lem:deriv_disp}; see also Remark~\ref{rem:deriv_3d} for the
	corresponding assumptions and additional regularity requirements.
	
	Quadratic functionals are among the most commonly used objective
	functionals in shape optimization and include, in particular, the
	compliance and the \(L^2\)-norm of the von Mises stress considered 
	in Section~\ref{sec:numerics}. More general nonlinear integrands 
	can also be treated, provided that the above regularity conditions 
	are fulfilled.
\end{remark}

\section{Numerical realization}\label{sec:numerics}
This part of the article is devoted to the numerical methods and 
algorithms we employ for solving the corresponding shape optimization 
problems. We adopt the level-set method to handle the shape optimization 
process. To demonstrate the efficacy of the proposed approach, we perform 
numerical simulations for the optimization of a bridge-type structure 
subject to elevated external temperatures. We consider two distinct 
optimization scenarios to evaluate the robustness of the algorithm: 
first, the minimization of the structural compliance under a prescribed 
volume constraint, and second, the minimization of the total volume 
subject to a threshold on the \(L_2\)-norm of the von Mises stress. 
Both cases are simulated using the temperature dependent properties 
of the GH4099 superalloy. The underlying boundary value problems are 
solved with the finite element solver FreeFem++, see \cite{hecht2012new}. 

\subsection{Level-set method}
To ensure an accurate geometric representation during the optimization 
process, we use the level-set method. Originally introduced for 
front-tracking problems in \cite{osher1988fronts}, its application 
to shape and topology optimization has been further developed in 
\cite{allaire2004structural}.

In the level-set method, the domain \(D\subset\Dbox\) is given by the subset of 
negative function values of a \emph{level-set function\/} \(\phi:\Dbox\to\mR\),
such that
\begin{equation*}
	\begin{cases}
		\phi(\bx)<0 \text{~~if~~} \bx \in D, \\
		\phi(\bx) = 0 \text{~~if~~} \bx \in \partial D, \\
		\phi(\bx) > 0 \text{~~if~~} \bx \in D_\text{box}\setminus\overline{D}.
	\end{cases}
\end{equation*}
The discretized evolution of the domain \(D\) in \(D_\text{box}\) 
is governed by the Hamilton–Jacobi equation, which is expressed by
\begin{equation*}
	\left\{\,
	\begin{aligned}
		&\phi_{n+1} = \phi_n - \iota_n\btheta_n\cdot\nabla\phi_n
		\text{~~in~~} \Dbox, \\[1ex]
		&\phi_0 = \phi_\text{{\rm in}}\text{~~in~~} \Dbox,
	\end{aligned}
	\right.
\end{equation*}
where \(\iota_n=\text{const}>0\) is the discretization step size, 
\(\phi_\text{{\rm in}}\) is a chosen initial level-set 
function, and \(\btheta_n\) is the velocity field.

The velocity field is determined through a constrained optimization 
framework, using the shape derivatives of both the objective 
functional and the constraints. Rather than employing a standard 
projected gradient scheme, we utilize the null-space gradient flow 
algorithm proposed by Feppon et al.~\cite{feppon2020null}, which 
has gained popularity in the field of numerical shape optimization. 
This algorithm starts from a point that does not satisfy the constraint, 
and seeks to orientate the direction of the deformation at each step 
in such a way that the constraint is fulfilled first, while the objective
is reduced if possible. 

We use for our numerical computations the Dapogny-Feppon 
implementation of the null-space optimization algorithm for a 
level-set based mesh evolution method. Note that this method 
works without reinitialization of the level-set function. For more details, 
we refer the reader to \cite{CRMATH_2023__361_G8_1267_0}.

\subsection{Computational model}
The external domain \(\Dbox\in \mR^2\) is chosen as the square 
of size \(1 \ [m^2]\). The mesh is automatically adapted 
towards the current geometry in the course of the optimization 
process, where the finite element size varies from \(h_{\min} 
= 0.01 \ [m]\) to \(h_{\max} = 0.02 \ [m]\). The initial mesh is presented 
in Figure~\ref{fig:init}.

\begin{figure}[hbt]
	\centering
	\begin{minipage}[hbt]{0.35\linewidth}
		\centering\includegraphics[width=1\linewidth]{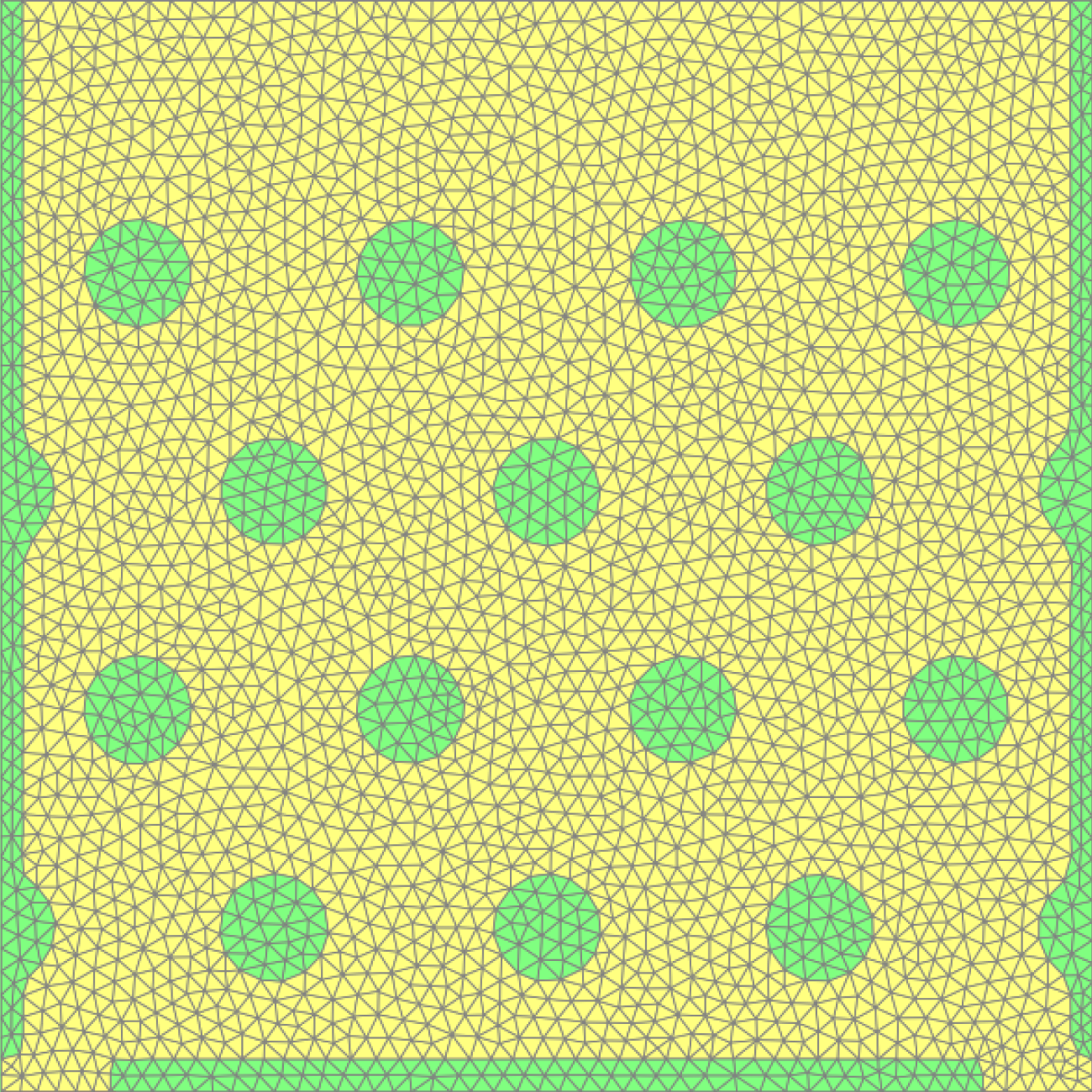} (a)
	\end{minipage}
	\hspace{10mm}
	\begin{minipage}[hbt]{0.35\linewidth}
		\centering\includegraphics[width=1\linewidth]{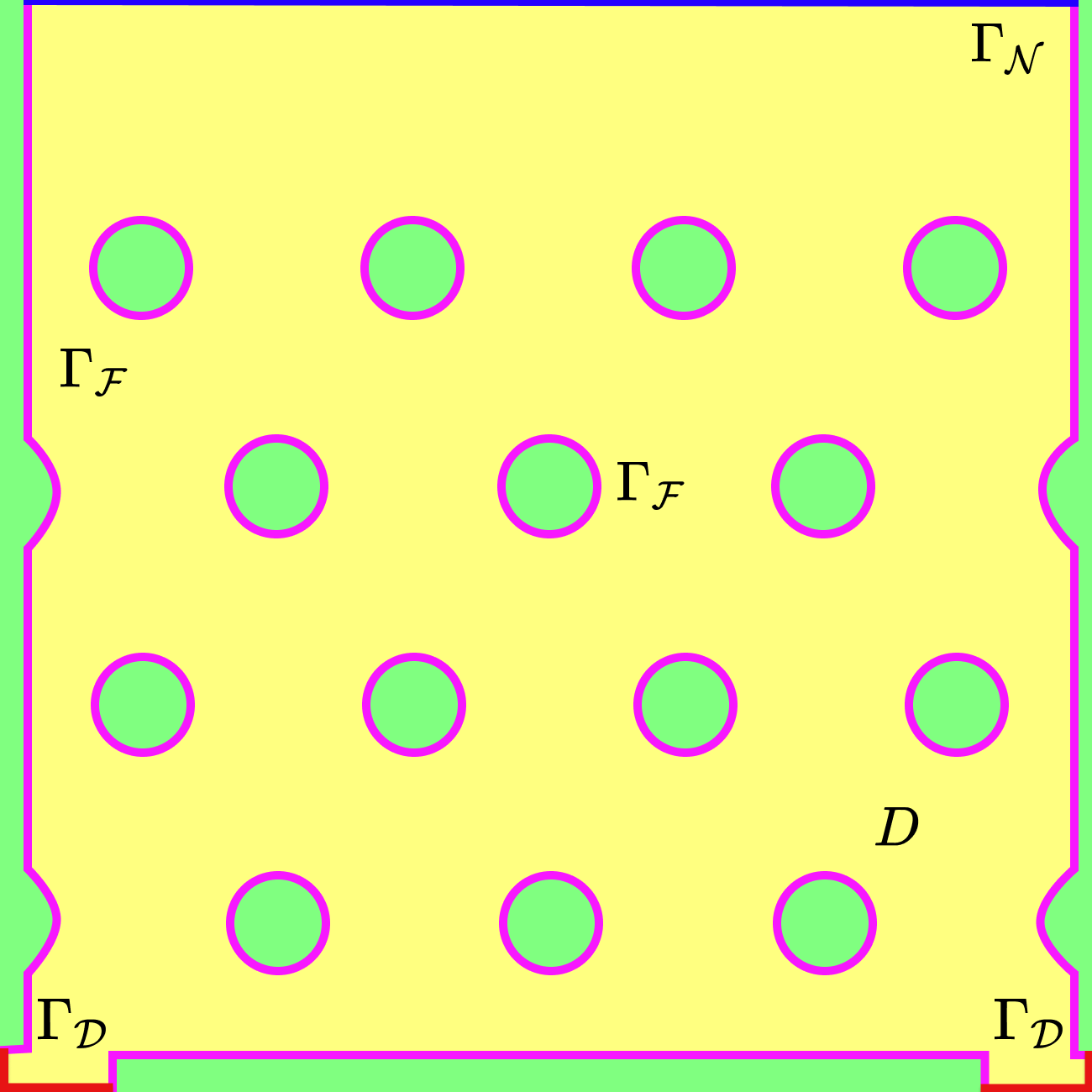} (b)
	\end{minipage}
	\caption{Initial setting. (a) Triangulation and (b) boundary conditions.}
	\label{fig:init}
\end{figure}

The final time is taken as \(\tf=1\). We solve the dynamic problems 
\eqref{eq:heat} and \eqref{eq:adjoint_temp} by using the Crank–Nicolson 
scheme. We apply a grading in the time discretization. To this end, 
we split the time interval \((0,t_f)\) into 20 subintervals, where the sizes 
of the first and last 5 subintervals decrease towards the boundaries with 
a factor of \(1/2\), i.e., our time steps look like \(h_t/2, h_t/4, \dots, h_t/32\). 
The intervals in the middle are of a fixed size \(h_t\). At each time step, 
we apply a fixed-point iteration to solve the corresponding nonlinear 
problem.

The material parameters employed in the thermoelasticity model are 
chosen to simulate the behavior of GH4099, a superalloy widely 
utilized in the aerospace industry \cite{tang2023topology}. Accordingly,
the specific heat capacity \(\widetilde{C}(T)\), the thermal conductivity 
\(k(T)\), the Young's modulus \(E(T)\), and the thermal expansion 
coefficient \(\alpha(T)\) are given by
\begin{align*}
	\widetilde{C}(T) &:= 429.46 + 0.277\cdot T - 1.67\cdot10^{-6} \cdot T^2 \ 
	[\text{J}\cdot\text{kg}^{-1}\cdot\text{\textcelsius}^{-1}], \\[1ex]
	k(T) &:= 8.7245 + 0.0183\cdot T + 3.14\cdot10^{-7} \cdot T^2 \ 
	[\text{W$\cdot$m}^{-1}\cdot\text{\textcelsius}^{-1}],\\[1ex]
	E(T) &:= 217.3966 + 0.0223\cdot T - 1.3741\cdot 10^{-4}\cdot T^2 \ 
	[\text{GPa}],\\[1ex]
	\alpha(T) &:= (11.96 + 0.0014\cdot T + 3.48\cdot10^{-6} \cdot T^2)
	\cdot10^{-6} \ [ \text{\textcelsius}^{-1}],
\end{align*}
while the mass density \(\widetilde{\rho}\) and the Poisson ratio 
\(\nu\) are independent of the temperature $T$:
\[
	\widetilde{\rho} := 8190 \ [\text{kg}\cdot\text{m}^{-2}], 
	\quad \nu := 0.3.
\]

We choose a high heat transfer coefficient \(\beta = 500 \ 
[\text{W$\cdot$m}^{-2}\cdot\text{\textcelsius}^{-1}]\) and consider 
the situation that there is  no thermal source inside the structure, 
i.e., \(Q = 0 \ [\text{W$\cdot$m}^{-3}]\). The initial temperature is 
taken as \(\Tin=0\ [\text{\textcelsius}]\). The external temperature 
is modelled as a smooth traveling front propagating through 
\(D_{\text{box}}\) in the \(x_2\)-direction (see Figure \ref{fig:temp_ext}).
Precisely, it is defined as
\[
\Tex(t,\bx) := \frac{1}{2}\widetilde T_{\text{ex}} 
\left( 1 + \tanh \left( \frac{t/t_f - x_2}{0.05} \right) \right)
\]
with \(\ \widetilde T_{\text{ex}} = 1000 \ [\text{\textcelsius}]\). 
Finally, the body force is \(\bf=(0,0) \ [\text{N$\cdot$m}^{-3}]\) 
and the surface force is \(\bg=(0,-0.5) \ [\text{GPa}]\). Note 
that this model simulates a case in which the structural 
component is located in a hot environment.

\begin{figure}[hbt]
	\centering
	\begin{minipage}[hbt]{0.3\linewidth}
		\centering\includegraphics[width=1\linewidth]{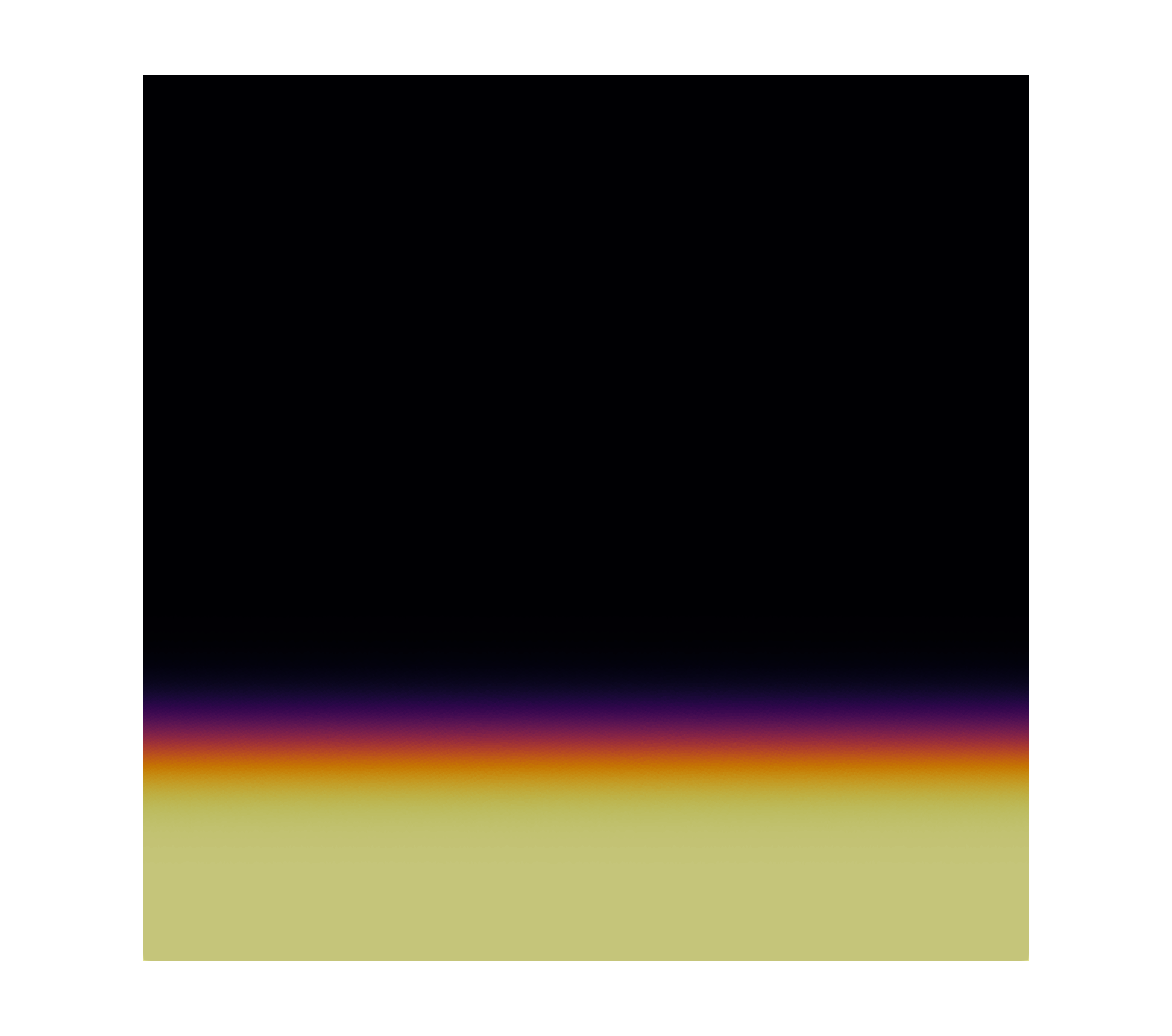} (a)
	\end{minipage}
	\begin{minipage}[hbt]{0.3\linewidth}
		\centering\includegraphics[width=1\linewidth]{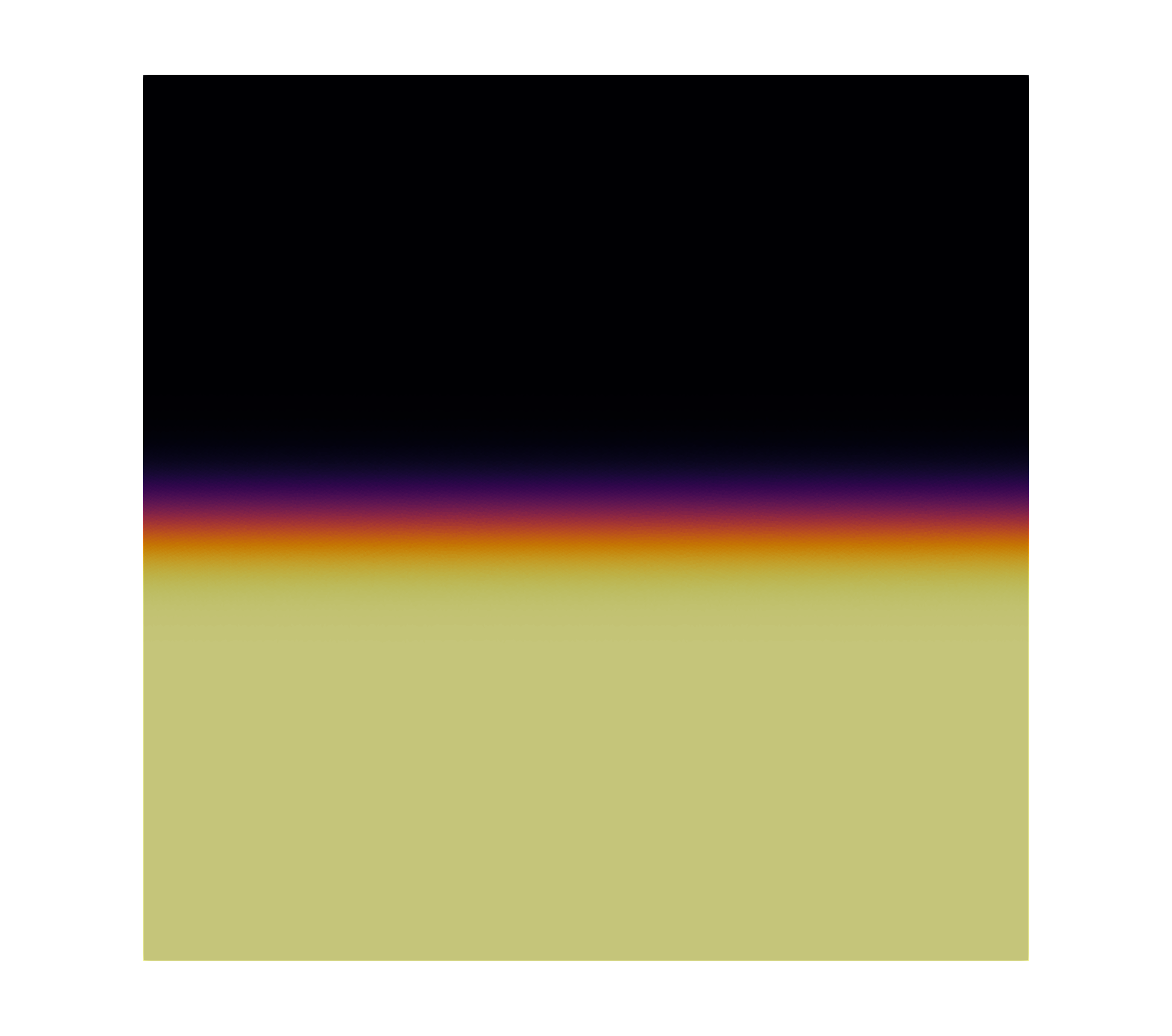} (b)
	\end{minipage}
	\begin{minipage}[hbt]{0.3\linewidth}
		\centering\includegraphics[width=1\linewidth]{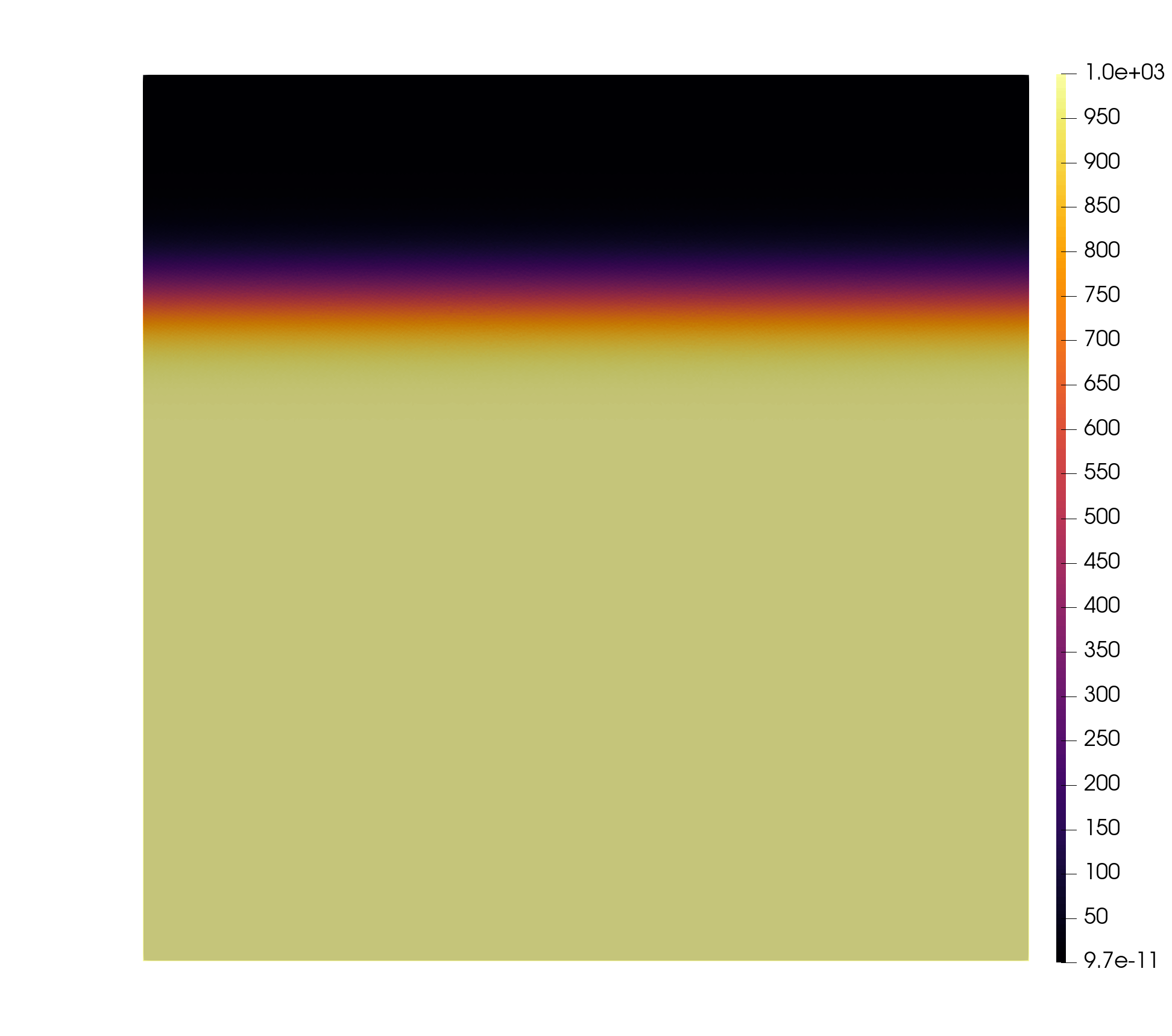} (c)
	\end{minipage}
	\caption{The external temperature at time (a): \(t=0.25\), 
	(b):  \(t=0.5\), (c): \(t=0.75\).}
	\label{fig:temp_ext}
\end{figure}

\subsection{Numerical results}
We next perform two numerical experiments to illustrate 
our method on one hand and to validate it on the other hand. 
To this end, we first minimize the compliance under a volume 
constraint. Afterwards, we minimize the volume under a constraint
on the $L^2$-norm of the von Mises stress. 

\subsubsection{Compliance minimization under a volume constraint}
We consider the minimization of the compliance 
\[
\F_\textrm{comp}(D) := \int_0^\tf\int_D \sigma(T_D, \bu_D) : 
	\varepsilon(\bu_D) \dx\dt
\]
subject to a constraint on the volume 
\[
\operatorname{Vol}(D) := \int_D \dx
\]
of the body, that is
\begin{equation}\label{eq:num1}
	\underset{D\subset\Dbox}{\text{minimize}}\quad 
	\F_\textrm{comp}(D)\quad\text{subject to}\quad 
	\operatorname{Vol}(D)\le \tau_1.
\end{equation}
Thus, according to Theorem \ref{th:shape_deriv}, the shape
derivatives are given by
\[
	\operatorname{Vol}'(D) = \int_{\GF} \btheta\cdot\bn\ds
\]
and
\begin{equation*}
	\begin{aligned}
		\F_\textrm{comp}'(D)\langle\btheta\rangle =\int_0^\tf \int_{\GF} \Big(
		\sigma(T_D, \bu_D)& : \varepsilon(\bu_D-\bw_D) + \bf\cdot\bw_D + QR_D \\[1ex] 
		&- \rho(T_D)\partial_t T_D R_D  - k(T_D)\bnabla T_D\cdot\bnabla R_D\\[1ex]
		&+\beta(\mathcal{H}-\beta/k(T_D))(\Tex - T_D)
		R_D\Big)\Big(\btheta\cdot\bn\Big)\ds\dt.
	\end{aligned}
\end{equation*}
Herein, \( T_D \in \Win\big((0,\tf),D\big)\) is the solution 
of \eqref{eq:heat} and \(\bu_D\in \HD((0,tf),D)^d\) 
is the solution of \eqref{eq:elasticity}, \(\bw_D\in \HD((0,tf),D)^d\) 
satisfies 
\begin{equation*}
	\left\{\;
	\begin{aligned}
		-\div(\mC(T_D):\varepsilon(\bw_D)) &= -2\div(\sigma(T_D, \bu_D))
		&&\text{in}\quad  (0,\tf)\times D, \\[1ex]
		(\mC(T_D):\varepsilon(\bw_D))\bn &=2\sigma(T_D, \bu_D) \bn  
		&&\text{on}\quad  (0,\tf)\times \GR,
	\end{aligned}
	\right.
\end{equation*}
while \(R_D\in \Wf\big((0,\tf),D\big)\) satisfies 
\begin{equation*}
	\left\{\;
	\begin{aligned}
		\rho(T_D)\partial_t R_D+\div\big(k(T_D) \bnabla R_D\big) 
		- \rho'(T_D)\partial_t T_D R_D& \\
		- k'(T_D)\bnabla T_D\cdot\bnabla R_D 
		&= \widetilde{Q} &&\text{in}\quad (0,\tf)\times D, \\[1ex]
		\big(k(T_D) \bnabla R_{D}\big)\cdot \bn  
		&= 0
		&&\text{on}\quad (0,\tf)\times\GN, \\[1ex]
		\big(k(T_D) \bnabla R_{D}\big)\cdot \bn + \beta R_D 
		&=  0
		&&\text{on}\quad (0,\tf)\times\GF, \\[1ex]
	\end{aligned}
	\right.
\end{equation*}
where 
\[
\widetilde{Q} := \sigma^\ast(T_D, \bu_D):\varepsilon(\bu_D-\bw_D) + \bB(T_D):\varepsilon(\bu_D-\bw_D)
\] 
and \( \sigma^\ast(T_D, \bu_D)\) is defined in \eqref{eq:deriv_disp_not}. 

By setting the threshold \(\tau_1=0.4\), we obtain the numerical 
results found in Figure~\ref{fig:num_compliance} after 150 gradient 
steps of the shape optimization algorithm. At the beginning, the 
algorithm aims to satisfy the constraint, so the body’s volume 
decreases while the compliance naturally increases. This is followed 
by a phase of minimizing the compliance under the satisfied constraint, 
during which the body changes its shape while maintaining its volume. 

\begin{figure}[htb]
	\centering
	\begin{minipage}[hbt]{0.45\linewidth}
		\centering\includegraphics[width=1\linewidth]{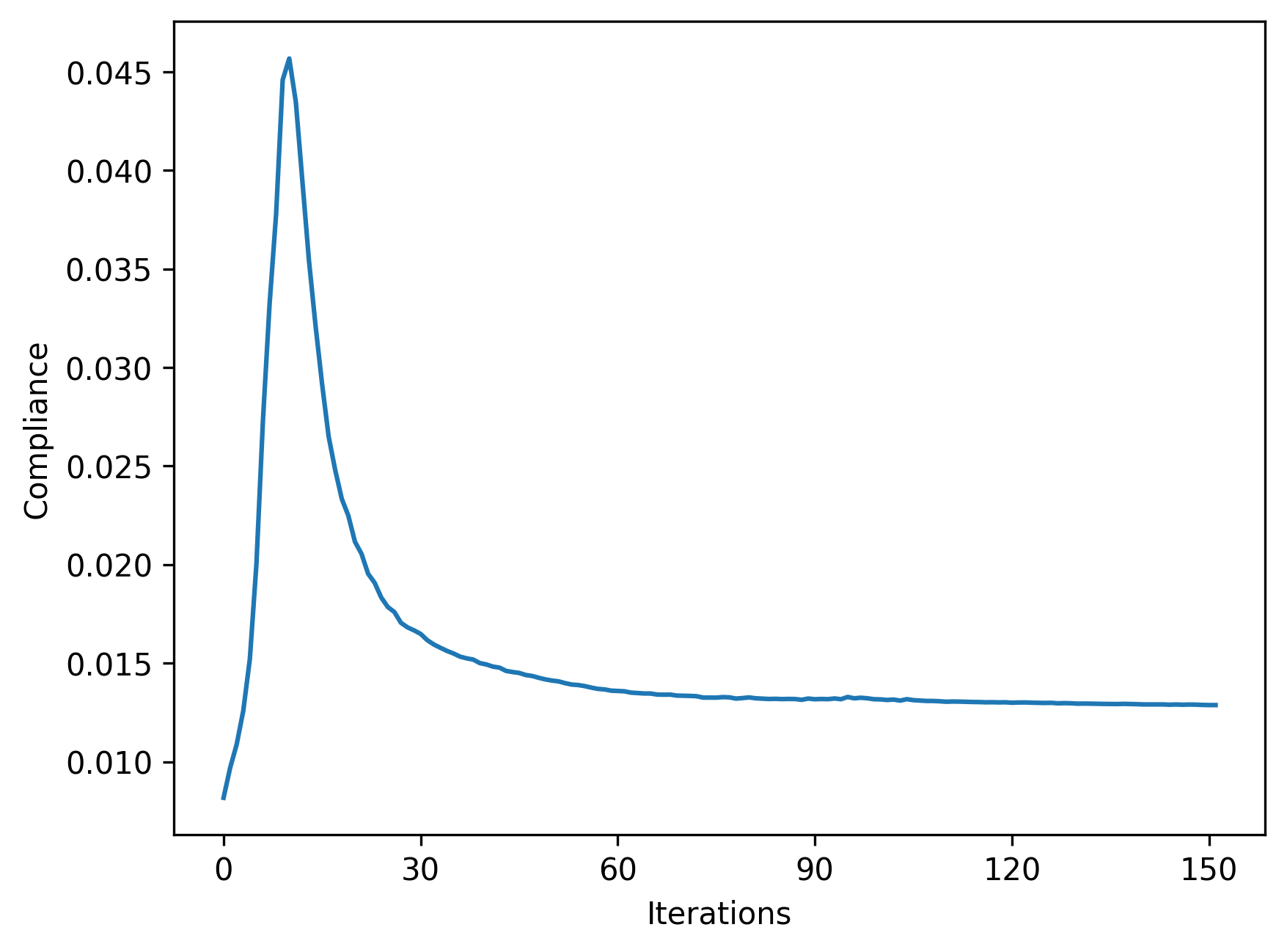} (a)
	\end{minipage}
	\begin{minipage}[hbt]{0.45\linewidth}
		\centering\includegraphics[width=1\linewidth]{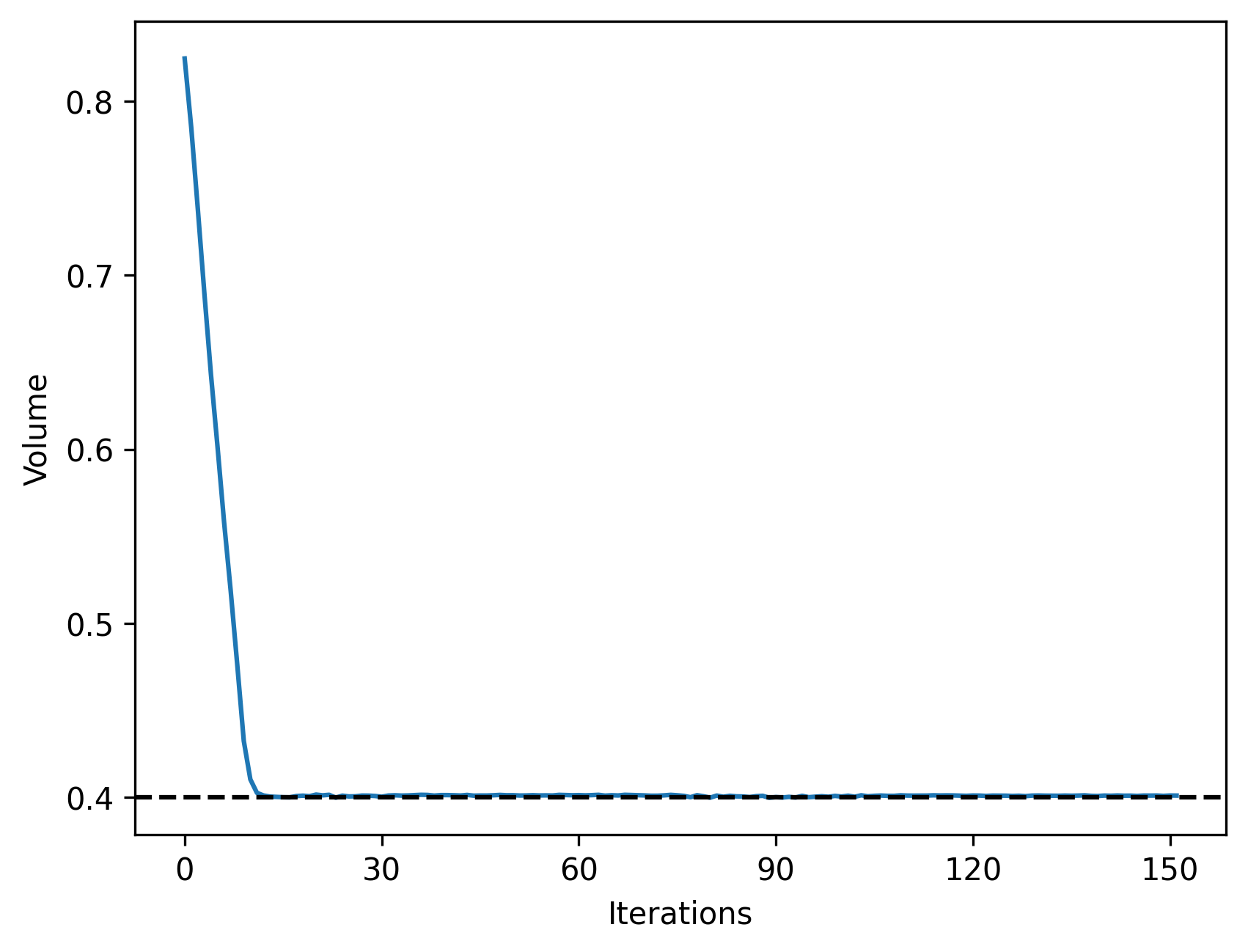} (b)
	\end{minipage}\\[2ex]
	\begin{minipage}[hbt]{0.45\linewidth}
		\centering\includegraphics[width=1\linewidth]{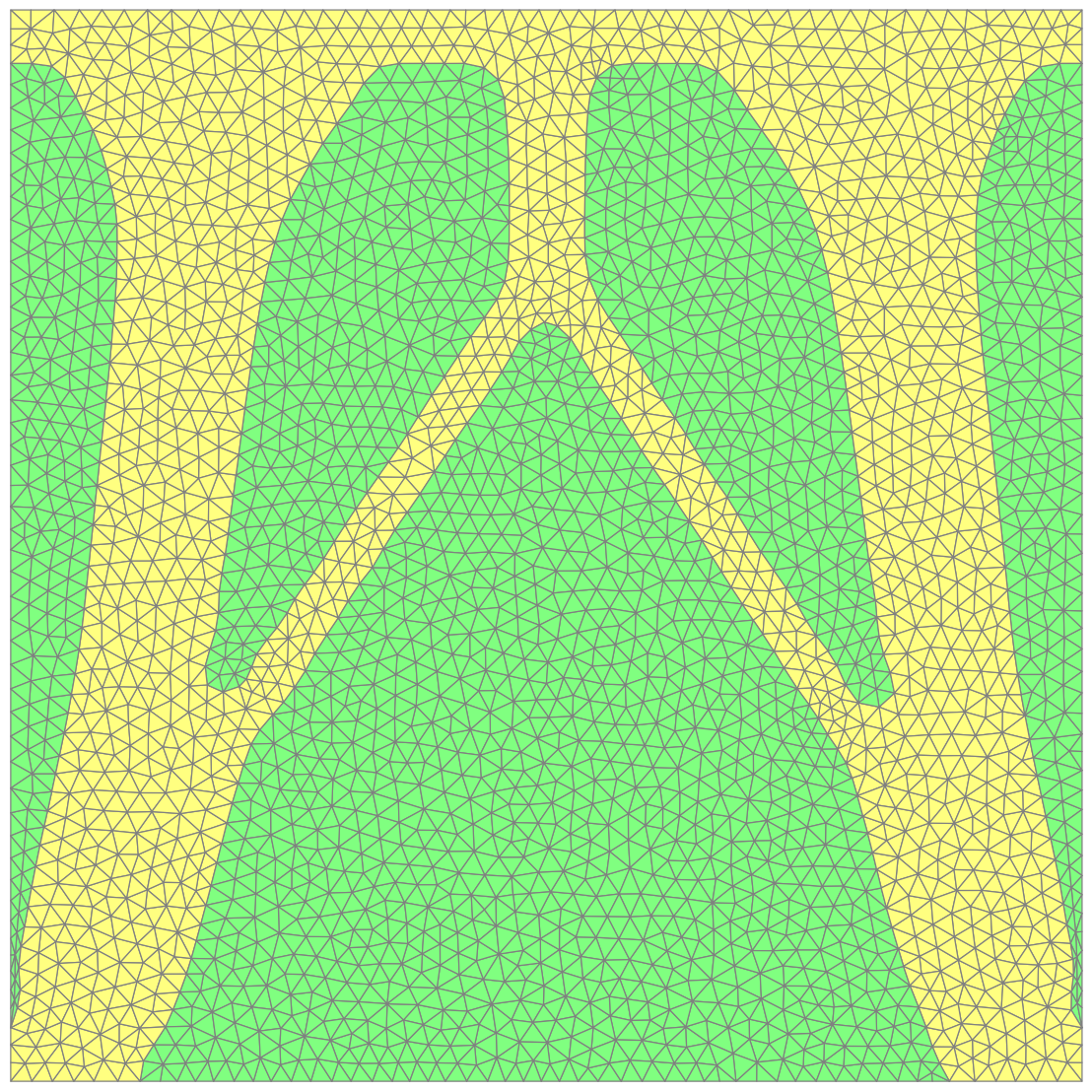} (c)
	\end{minipage}
	\caption{Experiment 1: Compliance minimization under a 
	volume constraint. Top row: convergence histories of (a) 
	the objective functional \(\F_\textrm{comp}(D)\) (compliance) 
	and (b) the constraint functional \(\operatorname{Vol}(D)\) 
	(volume). Bottom row: (c) final shape.}
	\label{fig:num_compliance}
\end{figure}

\subsubsection{Volume minimization under a von Mises stress constraint}
In our second experiment, we consider the minimization of the volume 
of the body subject to a constraint on the \(L^2\)-norm of the von Mises 
stress, which is referred to as the stress intensity and which is defined as
\[
	\sigma_\text{VM}(T, \bu) = \sqrt{\sigma_d(T, \bu):
	\sigma_d(T, \bu)},\quad\text{where}\quad\sigma_d(T, \bu)=\mu(T)
	\big(2\varepsilon(\bu) - \div(\bu)\bI\big).
\]
Hence, we consider the following shape optimization problem
\begin{equation}\label{eq:num2}
	\underset{D\subset\Dbox}{\text{minimize}}\quad 
	\operatorname{Vol}(D)\quad\text{subject to}\quad 
	\F_\textrm{VM}(D)\leq \tau_2,
\end{equation}
where
\[
	\F_\textrm{VM}(D) := \int_0^\tf\int_D|\sigma_\text{VM}(T_D, \bu_D)|^2\dx\dt =
	 \int_0^\tf\int_D \sigma_d(T_D, \bu_D) : \sigma_d(T_D,\bu_D) \dx\dt.
\]
According to Theorem \ref{th:shape_deriv}, the shape derivative 
is given by
\begin{equation*}
	\begin{aligned}
		\F_\textrm{VM}'(D)\langle\btheta\rangle = \int_0^\tf\int_{\GF} 
		\Big(\sigma_d(T_D, \bu_D) &: \sigma_d(T_D,\bu_D) 
		- \sigma(T_D,\bu_D) : \varepsilon(\bw_D) + \bf\cdot\bw_D  \\[1ex]
		&+ QR_D- \rho(T_D)\partial_t T_D R_D  - k(T_D)\bnabla T_D\cdot\bnabla R_D\\[1ex]
		&+\beta(\mathcal{H}-\beta/k(T_D))(\Tex - T_D)
		R_D\Big)\Big(\btheta\cdot\bn\Big)\ds\dt.
	\end{aligned}
\end{equation*}
Herein, \( T_D \in \Win\big((0,\tf),D\big)\) is the solution 
of \eqref{eq:heat} and \(\bu_D\in \HD((0,tf),D)^d\) 
is the solution of \eqref{eq:elasticity}, \(\bw_D\in \HD((0,tf),D)^d\) 
satisfies 
\begin{equation*}
	\left\{\;
	\begin{aligned}
		-\div\big(\mC(T_D):\varepsilon(\bw_D)\big) &= - 4\div\big(\mu(T_D)
		\sigma_d(T_D, \bu_D)\big)&&\text{in}\quad  (0,\tf)\times D, \\[1ex]
		\big(\mC(T_D):\varepsilon(\bw_D)\big)\bn &= 4\mu(T_D)
		\sigma_d(T_D, \bu_D) \bn &&\text{on}\quad  (0,\tf)\times \GR,
	\end{aligned}
	\right.
\end{equation*}
while \(R_D\in \Wf\big((0,\tf),D\big)\) satisfies 
\begin{equation*}
	\left\{\;
	\begin{aligned}
		\rho(T_D)\partial_t R_D+\div\big(k(T_D) \bnabla R_D\big) 
		- \rho'(T_D)\partial_t T_D R_D& \\
		- k'(T_D)\bnabla T_D\cdot\bnabla R_D
		&= \widetilde{Q}&&\text{in}\quad (0,\tf)\times D, \\[1ex]
		\big(k(T_D) \bnabla R_{D}\big)\cdot \bn  
		&=  0
		&&\text{on}\quad (0,\tf)\times\GN, \\[1ex]
		\big(k(T_D) \bnabla R_{D}\big)\cdot \bn + \beta R_D 
		&=  0
		&&\text{on}\quad (0,\tf)\times\GF, 		
	\end{aligned}
	\right.
\end{equation*}
where 
\[
\widetilde{Q} := \sigma^\ast(T_D,\bu_D):\varepsilon(\bw_D) 
+ \bB(T_D):\varepsilon(\bw_D) - 2\sigma_d^\ast(T_D,\bu_D):\sigma_d(T_D, \bu_D),
\] 
with \(\sigma^\ast_d(T_D, \bu_D)=\mu'(T_D)\big(2\varepsilon(\bu_D) - \div(\bu_D)\bI\big)\)
and \(\sigma^\ast(\bu_D,T_D)\) defined as in \eqref{eq:deriv_disp_not}. 

By setting the threshold \(\tau_2=1.5\), we obtain the numerical 
results found in Figure~\ref{fig:num_vmstress} after 150 gradient 
steps of the shape optimization algorithm. Since, in this case, the 
initial shape already satisfies the constraint on the \(L^2\)-norm 
of the von Mises, we proceed directly to minimizing the volume. 
As in the previous case, a decrease in the volume of the body 
increases the stress intensity. It causes oscillations around the 
threshold, so the algorithm adapts the shape to satisfy the 
constraint. The peaks visible in the \(L^2\)-norm of the von Mises 
stress graph represent changes in the topology of the body during 
the optimization process. 

\begin{figure}[hbt]
	\centering
	\begin{minipage}[hbt]{0.45\linewidth}
		\centering\includegraphics[width=1\linewidth]{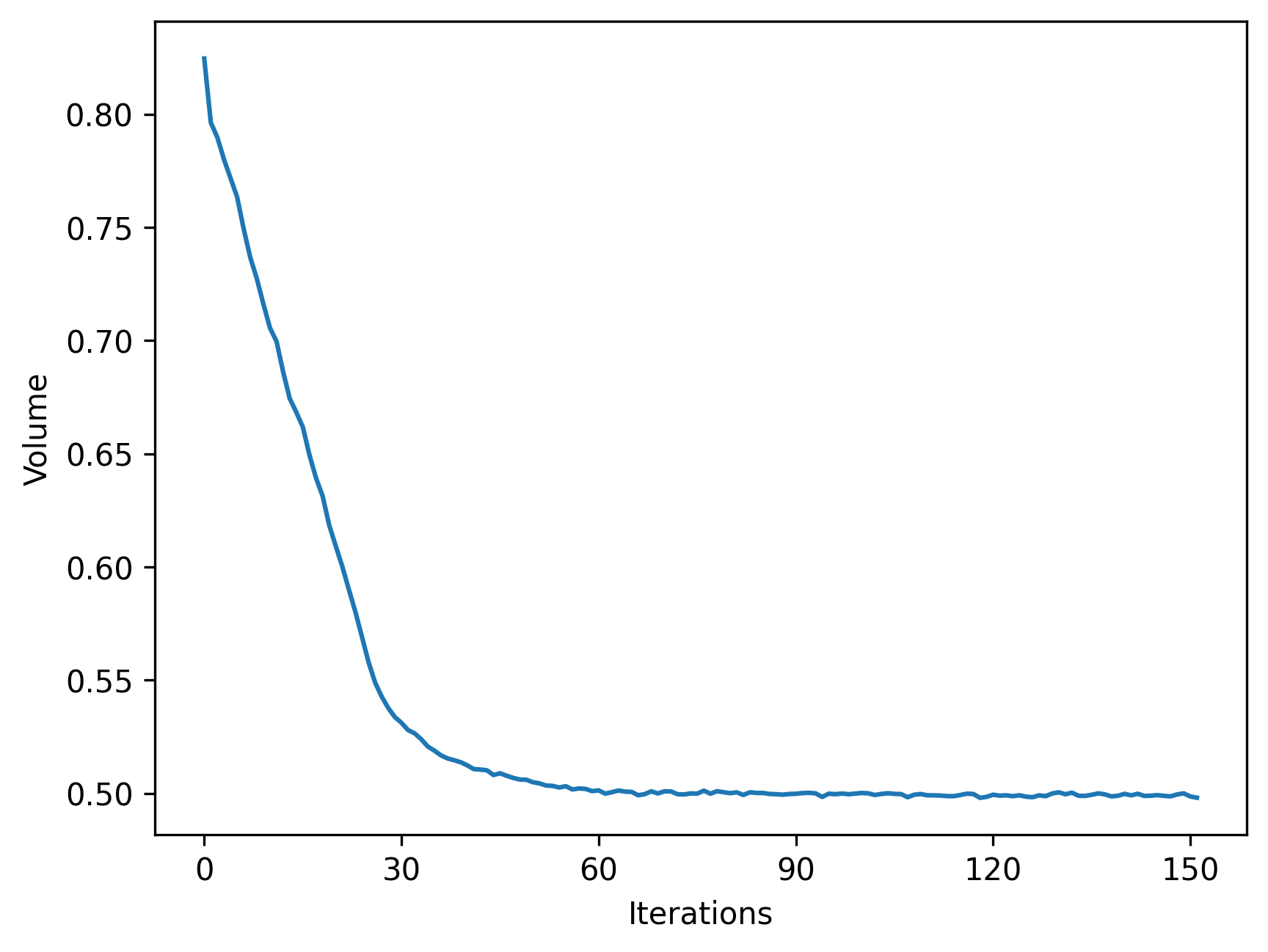} (a)
	\end{minipage}
	\begin{minipage}[hbt]{0.45\linewidth}
		\centering\includegraphics[width=1\linewidth]{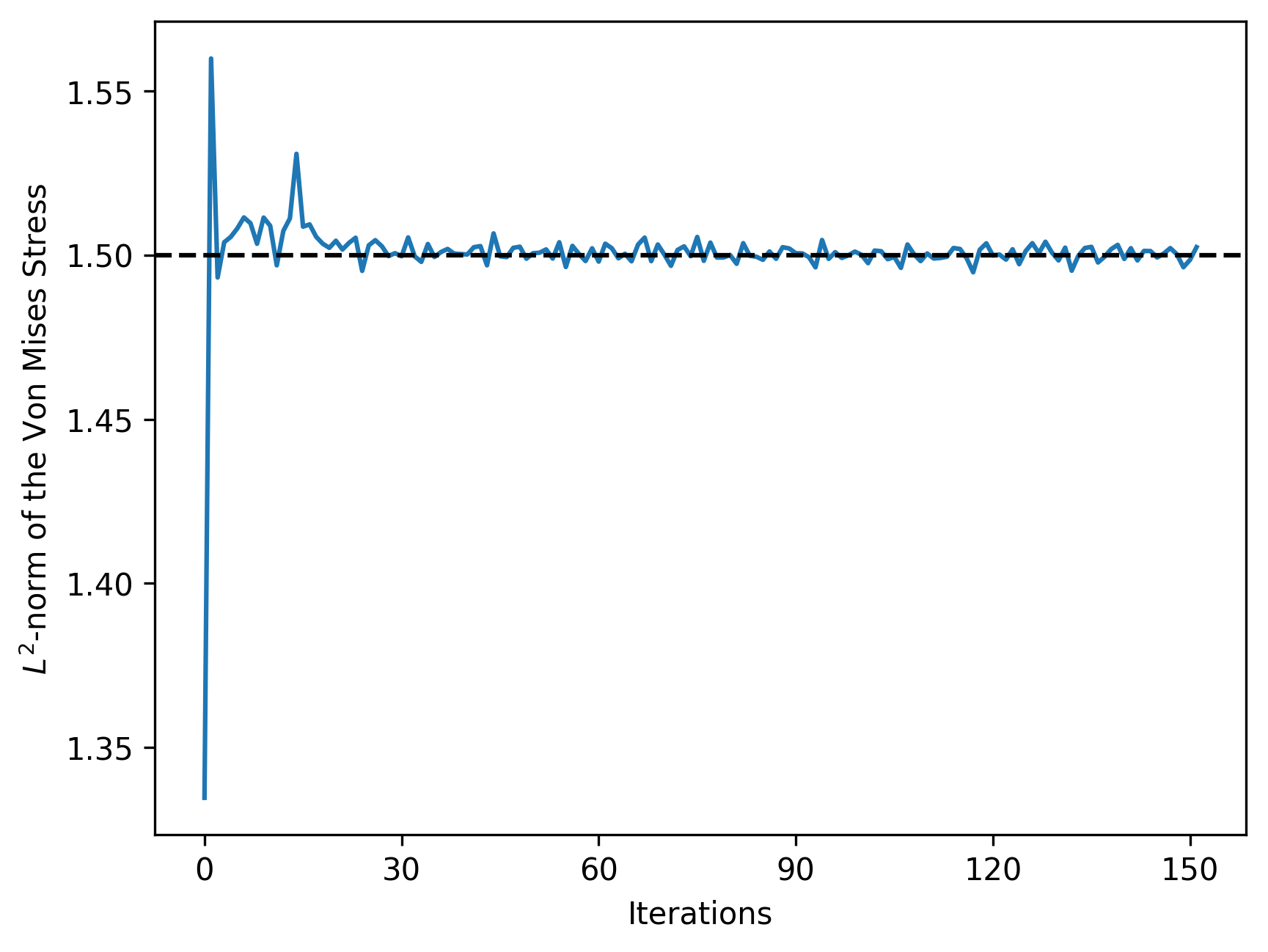} (b)
	\end{minipage}\\[2ex]
	\begin{minipage}[hbt]{0.45\linewidth}
		\centering\includegraphics[width=1\linewidth]{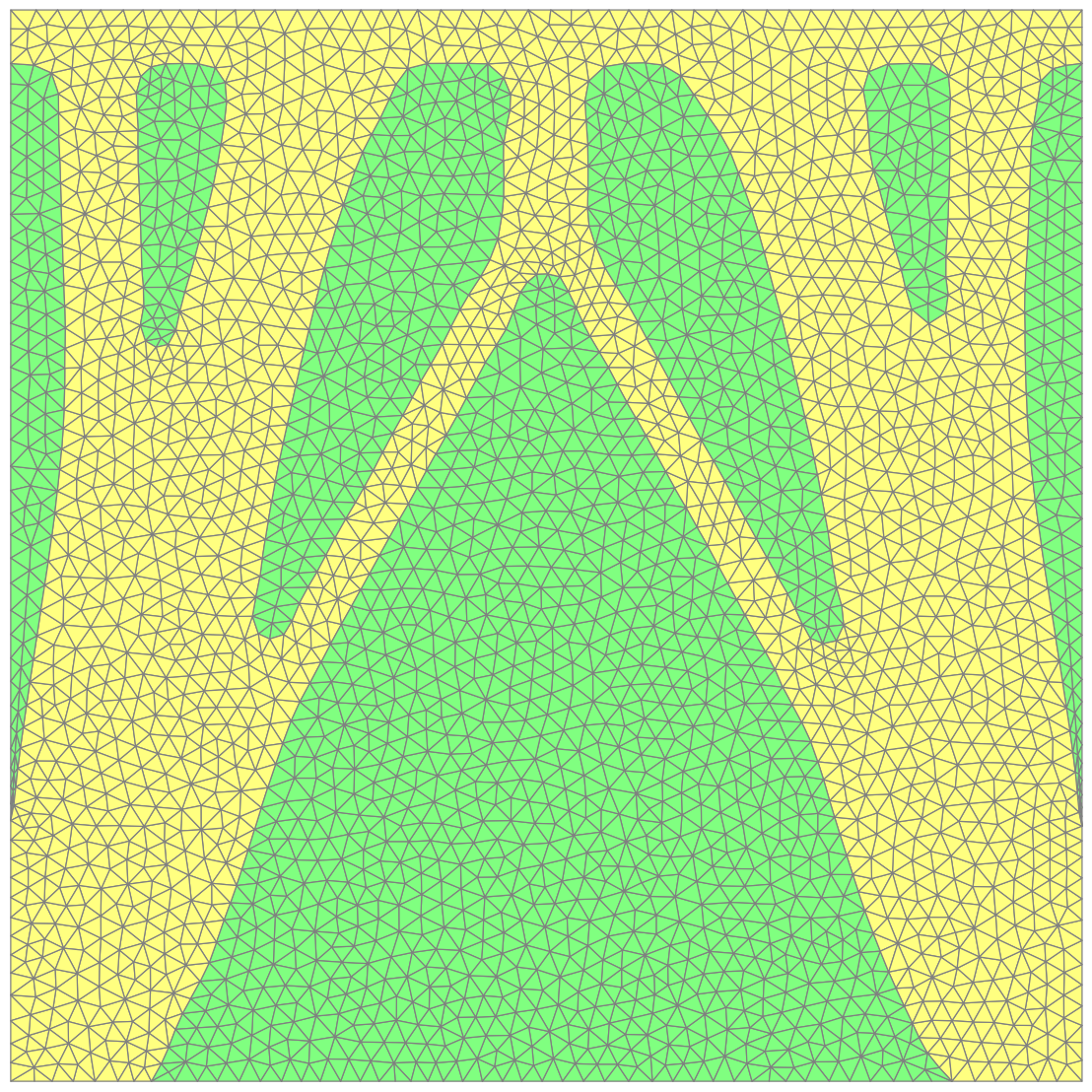} (c)
	\end{minipage}
	\caption{Experiment 2: Volume minimization under a von Mises
	stress constraint. Top row: convergence histories of (a)
	objective functional \(\operatorname{Vol}(D)\) (volume) and (b)
	constraint functional  \(\F_\textrm{VM}(D)\) (\(L^2\)-norm of the
	von Mises stress). Bottom row: (c) final shape.} 
	\label{fig:num_vmstress}
\end{figure}

\section{Conclusion}
\label{sec:conclusio}
The present article dealt with a gradient-based shape optimization 
problem for a nonlinear thermoelasticity model. Especially, we 
established the existence of the shape derivative for a generic shape 
functional of domain integral type. In our numerical realization, we 
combined the finite element method for solving the underlying 
boundary value problems with the level-set method to represent the 
domain. To illustrate and validate the proposed approach, we 
considered two numerical test cases: the minimization of the 
compliance under a volume constraint on one hand and the 
minimization of the volume under an \(L^2\)-norm constraint on the 
von Mises stress on the other hand. These experiments showed that 
the present approach is feasible and applicable in engineering practice.

\nolinenumbers 
\bibliographystyle{plain} 
\bibliography{references}
\end{document}